\documentclass{amsart}

\usepackage{amsmath,amssymb}
\usepackage{amsthm}
\usepackage{stmaryrd}
\usepackage{enumerate}
\usepackage{url}
\usepackage{wasysym}
\usepackage{color}
\usepackage[normalem]{ulem}

\newtheorem{theorem}{Theorem}[section]
\newtheorem{fact}[theorem]{Fact}
\newtheorem{lemma}[theorem]{Lemma}
\newtheorem{corollary}[theorem]{Corollary}
\newtheorem{proposition}[theorem]{Proposition}

\newtheorem{observation}[theorem]{Observation}

\theoremstyle{definition}
\newtheorem{definition}[theorem]{Definition}
\newtheorem{example}[theorem]{Example}

\newtheorem{remark}[theorem]{Remark}
\newtheorem{question}[theorem]{Question}
\newtheorem*{theorem*}{Theorem}

\def\tp{\operatorname{tp}}

\def\dfs{\operatorname{dfs}}

\def\Av{\operatorname{Av}}
\def\EEM{\operatorname{EEM}}
\def\EM{\operatorname{EM}}

\def\conv{\operatorname{conv}}
\def\cl{\operatorname{cl}}

\def\Ind{\setbox0=\hbox{$x$}\kern\wd0\hbox to 0pt{\hss$\mid$\hss}
\lower.9\ht0\hbox to 0pt{\hss$\smile$\hss}\kern\wd0}

\def\Notind{\setbox0=\hbox{$x$}\kern\wd0\hbox to 0pt{\mathchardef
\nn=12854\hss$\nn$\kern1.4\wd0\hss}\hbox to
0pt{\hss$\mid$\hss}\lower.9\ht0 \hbox to 0pt{\hss$\smile$\hss}\kern\wd0}

\title{Sequential Approximations for Types and Keisler Measures}

\author[K. Gannon]{Kyle Gannon}

\date{\today}

\address{Department of Mathematics\\
University of California, Los Angeles\\
Los Angeles, CA, 90095, USA}

\email{gannon@math.ucla.edu}

\begin{document}

\begin{abstract} This paper is a modified chapter of the author's Ph.D. thesis. We introduce the notions of sequentially approximated types and sequentially approximated Keisler measures. As the names imply, these are types which can be approximated by a sequence of realized types and measures which can be approximated by a sequence of ``averaging measures" on tuples of realized types. We show that both generically stable types (in arbitrary theories) and Keisler measures which are finitely satisfiable over a countable model (in NIP theories) are sequentially approximated. We also introduce the notion of a smooth sequence in a measure over a model and give an equivalent characterization of generically stable measures (in NIP theories) via this definition. In the last section, we take the opportunity to generalize the main result of \cite{GannNIP}.  
\end{abstract}


\maketitle
\vspace{-15pt} 

\section{Introduction} 
One of the joys of working in a metric space is that the closure of a set coincides with its \textit{sequential closure}. In particular, if $X$ is a metric space, $A$ is a subset of  $X$, and $b$ is in the closure of $A$, then there exists a sequence of elements in $A$ which converges to $b$. In \cite{Invariant}, Simon showed that global types which are finitely satisfiable over a countable model of a countable NIP theory admit a similar property. Let $T$ be a complete, first-order theory, $\mathcal{U}$ a monster model of $T$, and $M$ a small submodel of $\mathcal{U}$. Simon proved the following (\cite[Lemma 2.8]{Invariant}):
\begin{theorem}\label{sim:conv} Let $T$ be a countable NIP theory. Suppose $p$ is a type in $S_{x}(\mathcal{U})$ and $p$ is finitely satisfiable over $M$ where $|M| = \aleph_0$. Then there exists a sequence of points $(a_{i})_{i \in \omega}$ in $M^{x}$ such that $\lim_{i\to \infty} \tp(a_i/\mathcal{U}) = p$.
\end{theorem} 
One of the goals of this paper is to \textit{morally} generalize the proof of the above theorem in two different directions. By mimicking Simon's proof, we are able to prove the following,
\begin{enumerate}[(T$1$)]
    \item Let $T$ be any countable theory. Suppose $p$ is a type in $S_{x}(\mathcal{U})$ and $p$ is generically stable over $M$. Then there exists a sequence of points $(a_i)_{i \in \omega}$ in $M^{x}$ such that $\lim_{i \to \infty} \tp(a_i/\mathcal{U}) = p$.
    \item Let $T$ be a countable NIP theory. Suppose $\mu$ is a Keisler measure in $\mathfrak{M}_{x}(\mathcal{U})$ and $\mu$ is finitely satisfiable over $M$ where $|M| = \aleph_0$. Then there exists a sequence of points $(\overline{a}_{i})_{i \in \omega}$ in $(M^{x})^{< \omega}$ such that $\lim_{i \to \infty} \Av(\overline{a}_{i}) = \mu$. More explicitly, for any formula $\varphi(x)$ in $\mathcal{L}_{x}(\mathcal{U})$, we have that 
 \begin{equation*} \lim_{i \to \infty} \Av(\overline{a}_{i})(\varphi(x)) = \mu(\varphi(x)).
 \end{equation*} 
\end{enumerate}

The proofs of both of these theorems are slightly more \textit{enjoyable} than one would anticipate. For example, we already know many diverse and useful approximation theorems for measures in NIP theories (and some for generically stable types in arbitrary theories) and so one might expect that our proofs rely on composing approximation techniques. However, stringing together different approximation methods can result in an array with some kind of \textit{modes-of-convergence} problem. 

As stated previously, the technique used to prove both these theorems mimics the argument used in \cite[Lemma 2.8]{Invariant}. In the generically stable case, the set up is identical: Suppose $p$ is in $S_{x}(\mathcal{U})$ where $p$ is generically stable over $M$ and $I$ is a Morley sequence in $p$ over $M$. As in Simon's proof, we use both $M$ and $I$ to find an eventually indiscernible sequence of points in $M^{x}$ which converge to $p|_{MI}$. The \textit{eventual EM-type} of this sequence over $M$ is precisely $p^{(\omega)}|_{M}$. Using generic stability and compactness, we conclude that this sequence must converge to $p$. 

Our proof of the Keisler measure case is slightly more exotic since there is no standard notion of a ``Morley sequence in a Keisler measure". The proof we provide is \textit{essentially} done in first order model theory (with an important exceptional lemma following from Ben Yaacov's work on randomizations \cite{Ben}). We expect that there exists other proofs using other methods such as continuous model theory\footnote{In fact, after this paper was posted to arXiv, another proof was discovered by Khanaki using BFT on an infinite product space \cite{Khanaki3}.}. The proof we give here embraces the ideology first developed in \cite{HPS} and shows that this can be resolved by replacing the Morley sequence (in Simon's proof) by a \textit{smooth sequence in $\mu$ over $M$}. This provides more evidence for the intuition that smooth measures can play the role of realized types, at least in the NIP context. After constructing a countable model $N_{\omega}$ ``containing this sequence", we find a sequence of points in $(M^{x})^{<\omega}$ such that the corresponding average measures on these tuples converge to $\mu|_{N_{\omega}}$. After constructing an eventually indiscernible subsequence in this context, we are able to readapt most of Simon's proof technique by making use of known approximation theorems, symmetry properties, and some basic integration techniques.

It is interesting to note that one can give another equivalent characterization of generically stable measures in NIP theories using smooth sequences. This characterization highlights the connection between generically stable types and generically stable measures. Recall that a type $p$ is generically stable over a model $M$ if for every Morley sequence $(a_i)_{i \in \omega}$ in $p$ over $M$, $\lim_{i \to \infty} \tp(a_i/\mathcal{U}) = p$. We show that in an NIP theory, a measure $\mu$ is generically stable over a model $M$ if and only if for every \textit{smooth sequence} in $\mu$ over $M$, the limit of this sequence is precisely $\mu$. 

In addition to proving these theorems, we also introduce the classes of \textit{sequentially approximated measures} and \textit{sequentially approximated types}. These definitions can be seen as the \textit{global analogue} to Khanaki's definition of \textit{Baire 1 definability} for local types (see \cite{Khanaki2}). Sequentially approximated measures should be thought of as a ``halfway point" between finitely approximated measures and Keisler measures which are finitely satisfiable over a small model. For instance, we show that a Keisler measure is finitely approximated if and only if it is both definable and sequentially approximated (Proposition \ref{Mazur}) and sequentially approximated measures commute with definable measures (Proposition \ref{prop:com}). Sequentially approximated types remain a little more mysterious. We show that there exists a type such that its corresponding Keisler measure is sequentially approximated (even finitely approximated), but the type itself is not sequentially approximated (Proposition \ref{Gabe}).

In the last section, we consider connections to the local measure case and generalize the main result in \cite{GannNIP} (Theorem \ref{main:Gan}). Explicitly, the main result in \cite{GannNIP} demonstrates that if a formula $\varphi$ is NIP and $\mu$ is a $\varphi$-measure which is $\varphi$-definable and finitely satisfiable over a \textit{countable model}, then $\mu$ is $\varphi$-finitely approximated in said model. Here, we demonstrate that \textit{countable} can be replaced by \textit{small}.

This paper is structured as follows: In section 2, we discuss preliminaries. In section 3, we describe sequentially approximated measures and sequentially approximated types. In section 4, we show that if $p$ is generically stable over $M$, then $p$ is sequentially approximated over $M$. We also give some examples of types which are which are not sequentially approximated at the end of the section. In section 5, we show that if $T$ is a countable NIP theory, and $\mu$ is finitely satisfiable over a countable model $M$, then $\mu$ is sequentially approximated over $M$. We then give an equivalent characterization of generically stable measures in NIP theories using smooth sequences. In section 6, we generalize the main theorem in \cite{GannNIP}.

\subsection*{Acknowledgements} 
We would like to thank Gabriel Conant, James Hanson, Karim Khanaki, Pierre Simon and our Ph.D. defense committee Daniel Hoffmann, Anand Pillay, Sergei Starchenko, and Minh Chieu Tran for helpful discussions and comments. Thanks also to the referee for many helpful comments. This paper was also partially supported by the NSF research grant DMS-1800806 as well as the NSF CAREER grant DMS-1651321.

\section{Preliminaries} 
If $r$ and $s$ are real numbers and $\epsilon$ is a real number greater than $0$, then we write $r \approx_{\epsilon} s$ to mean $|r - s| < \epsilon$.

Fix $\mathcal{L}$ a countable language. Throughout this paper, we always have a countable, complete, first-order theory $T$ and a monster model $\mathcal{U}$ of $T$ in the background. The letters $M$ and $N$ will be used to denote small elementary submodels of $\mathcal{U}$. The letters $x,y,z$ will denote tuples of variables. If $A \subseteq \mathcal{U}$, we let $\mathcal{L}(A)$ be the collection of formulas with parameters from $A$ (modulo logical equivalence). A formula in $\mathcal{L}(A)$ is called an ``$\mathcal{L}(A)$-formula". If $x_0,...,x_k$ is a finite sequence of pairwise disjoint tuples of variables, we let $\mathcal{L}_{x_0,...,x_k}(A)$ be the collection of $\mathcal{L}(A)$-formulas with free variables in these tuples. We write $\mathcal{L}_{x_{0},...,x_{k}}(\emptyset)$ simply as $\mathcal{L}_{x_{0},...,x_{k}}$. If $(x_i)_{i \in \omega}$ is a countable sequence of pairwise distinct tuples of variables, we let $\mathcal{L}_{(x_i)_{i \in \omega}}(A) = \bigcup_{k \in \omega} \mathcal{L}_{x_0,...,x_k}(A)$. For a tuple $x$, let $A^{x}= \{(a_0,...,a_{|x|-1}): a_i \in A, i \leq |x|-1\}$. We let $(A^{x})^{<\omega}$ be the collection of all finite sequences of points in $A^{x}$. If we call $\varphi(x,y)$ a \textit{partitioned $\mathcal{L}_{x,y}(\mathcal{U})$-formula}, we treat $x$ as object variables and $y$ as parameter variables. The formula $\varphi^{*}(y,x)$ denotes the exact same formula as $\varphi(x,y)$, but with the roles exchanged for parameters and object tuples. Generally speaking, in any instance where we have multiple tuples of variables (e.g. $x$ and $y$, or $(x_1,x_2,x_3,...)$), we will always assume they are pairwise distinct without comment.

\textbf{Unlike similar papers about Keisler measures, we do not identify a type and its corresponding Keisler measure}. We let $S_{x}(A)$ denote the usual type space over $A$ and $\mathfrak{M}_{x}(A)$ the space of Keisler measures over $A$. We let $\mathfrak{M}_{(x_i)_{i \in \omega}}(\mathcal{U})$ be the collection of finitely additive probability measures on $\mathcal{L}_{(x_{i})_{i \in \omega}}(\mathcal{U})$. For any (tuple of) variable(s) $x$, and any subset $A \subseteq \mathcal{U}$, we have a map $\delta: S_{x}(A) \to \mathfrak{M}_{x}(A)$ via $\delta(p) = \delta_{p}$ where $\delta_{p}$ is the \textit{Dirac measure at the type $p$}. We sometimes refer to $\delta_{p}$ as the \textit{corresponding Keisler measure} of $p$. If $\overline{a} = (a_1,...,a_n)$ is a sequence of points in $\mathcal{U}^{x}$, then we let $\Av(\overline{a})$ be the associated average measure in $\mathfrak{M}_{x}(\mathcal{U})$. Explicitly, for any $\psi(x) \in \mathcal{L}_{x}(\mathcal{U})$, we define
 \begin{equation*}\Av(\overline{a})(\psi(x)) = \frac{|\{1\leq i \leq n: \mathcal{U} \models \psi(a_i)\}|}{n}.
 \end{equation*} 

\subsection{Basics of convergence} Recall that if $A \subseteq \mathcal{U}$, then both $S_{x}(A)$ and $\mathfrak{M}_{x}(A)$ carry a natural compact Hausdorff topology. For $S_{x}(A)$, we have the usual Stone space topology. Similarly, $\mathfrak{M}_{x}(A)$ admits a compact Hausdorff topology. There are two ways to describe this topology. First, this topology is the topology induced from the compact Hausdorff space $[0,1]^{\mathcal{L}_{x}(A)}$ where we identify each measure with the obvious map from $\mathcal{L}_{x}(A)$ to $[0,1]$. This topology on $\mathfrak{M}_{x}(A)$ can also be described as the coarsest topology such that for any continuous function $f: S_{x}(A) \to \mathbb{R}$, the map $\int f : \mathfrak{M}_{x}(A) \to \mathbb{R}$ is continuous.

We will routinely need to keep track of which sets of parameters our types and measures are converging over. Hence, we establish the following conventions. 
\begin{definition} Fix $A \subseteq \mathcal{U}$, $p \in S_{x}(A)$ and $\mu \in \mathfrak{M}_{x}(A)$. 
\begin{enumerate}[$(i)$]
\item We say that a sequence of types $(p_{i})_{i \in \omega}$, where each $p_i$ is in $S_{x}(A)$, \textbf{converges} to $p$ if it converges in the Stone space topology on $S_{x}(A)$, which we write as ``$\lim_{i \to \infty} p_i = p$ in $S_{x}(A)$" or simply as ``$\lim_{i \to \infty} p_i = p$" when the underlying space is obvious. We recall that $\lim_{i \to \infty} p_i = p$ if for every $\psi(x) \in p$, there exists some natural number $N_{\psi}$ such that for any $n > N_{\psi}$, $\psi(x) \in p_n$.
\item We say that a sequence of measures $(\mu_i)_{i \in \omega}$, where each $\mu_{i}$ is in $\mathfrak{M}_{x}(A)$, \textbf{converges} to $\mu$ if this sequence converges in the compact Hausdorff topology on $\mathfrak{M}_{x}(A)$, which we write as ``$\lim_{i \to \infty} \mu_{i} = \mu$ in $\mathfrak{M}_{x}(A)$" or simply as ``$\lim_{i \to \infty} \mu_i = \mu$" when there is no possibility of confusion. Notice that $\lim_{i \to \infty} \mu_i = \mu$ if for every $\psi(x) \in \mathcal{L}_{x}(A)$ and $\epsilon >0$, there exists some natural number $N_{\psi,\epsilon}$ such that for any $n > N_{\psi,\epsilon}$,
\begin{equation*}
|\mu_{n}(\psi(x)) - \mu(\psi(x))| < \epsilon.
\end{equation*}
\end{enumerate}
\end{definition} 
We now observe the relationship between finitely satisfiable types and measures and topological closure in their respective spaces.

\begin{fact}\label{Avcls} Suppose $p \in S_{x}(\mathcal{U})$, $\mu \in \mathfrak{M}_{x}(\mathcal{U})$ and $M \prec \mathcal{U}$. Assume that $p$ and $\mu$ are finitely satisfiable over $M$. Then the following are true.
\begin{enumerate}[($i$)]
\item The type $p$ is in the closure of $\{tp(a/\mathcal{U}): a \in M^{x}\}$ in $S_{x}(\mathcal{U})$. 
\item The associated Keisler measure $\delta_{p}$ is in the closure of $\{\delta_{a}: a \in M^{x}\}$ in $\mathfrak{M}_{x}(\mathcal{U})$. 
\item The measure $\mu$ is in the closure of 
\begin{equation*}
\Big\{\sum_{i=1}^{n} r_i \delta_{a_i}: n \in \mathbb{N}, r_i > 0, \sum_{i=1}^{n} r_i =1, a_i \in M^{x}\Big\}
\end{equation*}
 in $\mathfrak{M}_{x}(\mathcal{U})$.
\item The measure $\mu$ is in the closure of $\{\Av(\overline{a}): \overline{a} \in (M^{x})^{<\omega}\}$ in $\mathfrak{M}_{x}(\mathcal{U})$. 
\end{enumerate} 
\end{fact} 

We remark that the proof of $(i)$ is a standard exercise and the proof of $(ii)$ follows directly from $(i)$. A proof of $(iii)$ can be found at \cite[Proposition 2.11]{ChGan} and $(iv)$ follows directly from $(iii)$. 

\subsection{Types} 

We recall some basic definitions and facts about special kinds of types (e.g. generically stable types). Our notion of an \textit{EM-type} is not defined in complete generality since we are only concerned with countable sequences in this paper.

\begin{definition} Let $(a_i)_{i \in \omega}$ be a sequence of points in $\mathcal{U}^{x}$ and let $B \subseteq \mathcal{U}$. Then the \textbf{Ehrenfeucht-Mostowski type} or \textbf{EM-type} of the sequence $(a_{i})_{i \in \omega}$ over $B$, denoted $\EM((a_{i})_{i \in \omega }/B)$, is the following partial type: 
\begin{equation*}
 \{\varphi(x_0,...,x_k) \in \mathcal{L}_{(x_i)_{i \in \omega}}(B): \mathcal{U} \models \varphi(a_{i_{0}},...,a_{i_{k}}) \text{ for any } i_0 <...<i_{k} \}.
\end{equation*} 
We remark that this partial type corresponds to a subset of $S_{(x_{i})_{i \in \omega}}(B)$.
\end{definition} 

\begin{observation} It is clear from the definition above that for any sequence of points $(a_{i})_{i \in \omega}$ in $\mathcal{U}^{x}$ and any 
$B \subseteq \mathcal{U}$, the type $\EM((a_{i})_{i \in \omega }/B)$ is complete if and only if the sequence $(a_{i})_{i \in \omega}$ is indiscernible over $B$.
\end{observation}

The general notion of a \textit{generically stable type} was introduced by Pillay and Tanovi\'{c} in \cite{PiTa}. The definition of a generically stable type provided below was proved to be equivalent in \cite{CoGan} (see Proposition 3.2). We also provide the definition of a $\dfs$ type which will be important throughout this paper. In general, the class of $\dfs$ types strictly contains the class of generically stable types. 

\begin{definition} Suppose that $p \in S_{x}(\mathcal{U})$. 
\begin{enumerate}[$(i)$]
\item We say that $p$ is \textbf{dfs} if there exists a small model $M \prec \mathcal{U}$ such that $p$ is both definable and finitely satisfiable over $M$. In this case, we say that $p$ is \textbf{dfs over $M$}. 
\item We say that $p$ is \textbf{generically stable} if there exists a small model $M \prec \mathcal{U}$ such that $p$ is invariant over $M$ and for any Morley sequence $(a_i)_{i \in \omega}$ in $p$ over $M$, we have that $\lim_{i \to \infty} \tp(a_i/\mathcal{U}) = p$. In this case, we say that $p$ is \textbf{generically stable over $M$}. 
\end{enumerate}
\end{definition} 
Finally, we provide a collection of standard facts about these classes of types.

\begin{fact}\label{gfs:facts} Let $p$ be in $S_{x}(\mathcal{U})$ and $M \prec \mathcal{U}$.
\begin{enumerate}[$(i)$]
\item If $p$ is generically stable over $M$, then $p$ is $\dfs$ over $M$ $($\cite[Proposition 1]{PiTa}$)$.
\item If $p$ is $\dfs$ over $M$, then any Morley sequence in $p$ over $M$ is totally indiscernible over $M$ $($\cite[Proposition 3.2]{HP}, proof does not use NIP$)$.
\item If $p$ is generically stable/$\dfs$ over $M$ and $M_0$-invariant, then $p$ is respectively generically stable/$\dfs$ over $M_0$ $($generically stable case follows from $(i)$ of \cite[Proposition 1]{PiTa}; $\dfs$ case can be found in \cite[Lemma 2.8]{Sibook}$)$.
\item $($T is countable$)$ If $p$ is generically stable/$\dfs$ over $M$, there exists an elementary submodel $M_0$ such that $|M_0| = \aleph_0$ and $p$ is generically stable/$\dfs$ over $M_0$ $($Easy to check from $(iii)$$)$.
\item $($T is NIP$)$ If $p$ is $\dfs$ over $M$ then $p$ is generically stable over $M$ $($e.g. \cite[Theorem 2.29]{Sibook}$)$.
\end{enumerate}
\end{fact} 

\subsection{Keisler measures} In this subsection, we will briefly recall some important definitions and facts about these measures. As with any paper about Keisler measures, we provide the following \textit{standard atlas}. 
\begin{definition} Let $\mu \in \mathfrak{M}_{x}(\mathcal{U})$.
\begin{enumerate}[($i$)]
\item $\mu$ is \textbf{invariant} if there exists a model $M \prec \mathcal{U}$ such that for every partitioned $\mathcal{L}$-formula $\varphi(x,y)$ and $b,b' \in \mathcal{U}^{y}$ such that $b \equiv_{M} b'$, $\mu(\varphi(x,b)) = \mu(\varphi(x,b'))$. In this case, we say that $\mu$ is \textbf{$M$-invariant} or \textbf{invariant over $M$}.
\item If $\mu$ is invariant over $M$, then for every partitioned $\mathcal{L}(M)$-formula $\varphi(x,y)$, we can define the map $F_{\mu,M}^{\varphi}:S_{y}(M) \to [0,1]$ via $F_{\mu,M}^{\varphi}(q) = \mu(\varphi(x,b))$ where $b \models q$. When $M$ is obvious we will simply write $F_{\mu,M}^{\varphi}$ as $F_{\mu}^{\varphi}$. 
\item $\mu$ is \textbf{Borel-definable} if there exists a model $M \prec \mathcal{U}$ such that $\mu$ is $M$-invariant and for every partitioned $\mathcal{L}$-formula $\varphi(x,y)$, the map $F_{\mu,M}^{\varphi}$ is Borel. In this case, we say that $\mu$ is \textbf{Borel-definable over $M$}.
\item $\mu$ is \textbf{definable} if there exists a model $M \prec \mathcal{U}$ such that $\mu$ is $M$-invariant and for every partitioned $\mathcal{L}$-formula $\varphi(x,y)$, the map $F_{\mu,M}^{\varphi}$ is continuous. In this case, we say that $\mu$ is \textbf{$M$-definable} or \textbf{definable over $M$}.
\item $\mu$ is \textbf{finitely satisfiable over a small model} if there exists $M \prec \mathcal{U}$ such that for every formula $\varphi(x) \in \mathcal{L}_{x}(\mathcal{U})$, if $\mu(\varphi(x)) > 0$ then there exists $a \in M^{x}$ such that $\mathcal{U} \models \varphi(a)$. In this case, we say that $\mu$ is \textbf{finitely satisfiable over $M$}.
\item $\mu$ is \textbf{finitely approximated} if there exists a model $M \prec \mathcal{U}$ such that for every partitioned $\mathcal{L}$-formula $\varphi(x,y)$ and every $\epsilon > 0$, there exists $\overline{a} \in (M^{x})^{<\omega}$ such that
\begin{equation*}
\sup_{b \in \mathcal{U}^{y}} |\mu(\varphi(x,b)) - \Av(\overline{a})(\varphi(x,b))| < \epsilon. 
\end{equation*}
In this case, we say that $\mu$ is \textbf{finitely approximated over $M$}. 
\item $\mu$ is \textbf{smooth} if there exists a model $M \prec \mathcal{U}$ such that for any $\lambda \in \mathfrak{M}_{x}(\mathcal{U})$ if $\lambda|_{M} = \mu|_{M}$, then $\lambda = \mu$. If this is the case, we say that $\mu$ is \textbf{smooth over $M$}.
\end{enumerate}
\end{definition} 

We now provide a collection of basic facts. Statements $(i)$, $(iii)$, $(iv)$, and $(v)$ in Fact \ref{KM:imp} are relatively straightforward to prove and so we leave them as exercises.

\begin{fact}\label{KM:imp} Assume that $T$ is any theory and $\mu \in \mathfrak{M}_{x}(\mathcal{U})$ with $M \prec \mathcal{U}$. 
\begin{enumerate}[$(i)$]
\item If $\mu = \Av(\overline{a})$ for some $\overline{a} \in (M^{x})^{<\omega}$, then $\mu$ is smooth over $M$.
\item If $\mu$ is smooth over $M$, then $\mu$ is finitely approximated over $M$, $($e.g. \cite[Proposition 7.10]{Sibook}$)$.
\item If $\mu$ is finitely approximated over $M$, then $\mu$ is both definable and finitely satisfiable over $M$.
\item If $\mu$ is definable or finitely satisfiable over $M$, then $\mu$ is $M$-invariant. 
\item The measure $\mu$ is definable over $M$ if and only if for every partitioned $\mathcal{L}(M)$-formula $\varphi(x,y)$ and for every $\epsilon > 0$, there exists formulas $\psi_{1}(y),...,\psi_{n}(y) \in \mathcal{L}_{y}(M)$ and real numbers $r_1,...,r_n \in [0,1]$ such that
\begin{equation*}
\sup_{q \in S_{y}(M)} | F_{\mu,M}^{\varphi}(q) - \sum_{i=1}^{n} r_i \mathbf{1}_{\psi_i(y)}(q)| < \epsilon. 
\end{equation*} 
where $\mathbf{1}_{\psi_{i}(y)}$ is the characteristic function of the clopen set $[\psi_{i}(y)]$. 
\end{enumerate}
Moreover, if $T$ is NIP then the following also hold. 
\begin{enumerate}[$(vi)$]
\item If $\mu$ is invariant over $M$, then $\mu$ is Borel-definable $M$ $($e.g. \cite[Proposition 7.19]{Sibook}$)$. 
\item Any measure $\mu$ is definable and finitely satisfiable over $M$ if and only if $\mu$ is finitely approximated over $M$ $($\cite[Proposition 3.2]{HPS}$)$.
\item Every measure has a ``smooth extension". In particular, for any given $M \prec \mathcal{U}$ and $\mu \in \mathfrak{M}_{x}(\mathcal{U})$, there exists some $N$ such that $M \prec N \prec \mathcal{U}$ and a measure $\lambda \in \mathfrak{M}_{x}(\mathcal{U})$ such that $\lambda$ is smooth over $N$ and $\lambda|_{M} = \mu|_{M}$ $($\cite[Lemma 2.2]{HPS}$)$. 
\end{enumerate}
\end{fact} 

\begin{proposition}[T is countable]\label{m:countable} If $\mu$ is definable, finitely approximated, smooth or $\dfs$, then there exists a countable model $M_0$ such that $\mu$ is definable, finitely approximated, smooth or $\dfs$ over $M_0$ $($respectively$)$. 
\end{proposition} 

\begin{proof} We notice that the properties of definability and smoothness only require the existence of $\aleph_0$-many $\mathcal{L}(M)$-formulas (by \cite[Lemma 2.3]{HPS} and (v) of Fact \ref{KM:imp} respectively). If we choose an elementary submodel $M_0$ of $M$ containing the parameters from these formulas, then $\mu$ will have the desired property over $M_0$. Finitely approximated measures only require the existence of $\aleph_0$-many elements of $M$. Choosing an elementary submodel $M_0$ of $M$ with these elements demonstrates that $\mu$ is finitely approximated over $M_0$.

Finally, if $\mu$ is $\dfs$ then $\mu$ is definable over a countable model $M_{0}$. In particular, $\mu$ is invariant over $M_0$ and so $\mu$ is also finitely satisfiable over $M_0$ by the same argument as in \cite[Proposition 4.13]{GannNIP}.
\end{proof} 

\begin{remark} Assuming $T$ is countable, there are measures (even types) which are finitely satisfiable over a small submodel, but are not finitely satisfiable over a countable submodel. See Proposition \ref{omega} and Remark \ref{example:coheir1} for an explicit example. 
\end{remark} 

\begin{definition}Let $\mu \in \mathfrak{M}_{x}(\mathcal{U})$, $\nu \in \mathfrak{M}_{y}(\mathcal{U})$ and assume that $\mu$ is Borel-definable over $M$. Then we define the \textbf{Morley product} of $\mu$ and $\nu$ (denoted $\mu \otimes \nu)$ is the unique Keisler measure in $\mathfrak{M}_{x,y}(\mathcal{U})$ with the following property: for any formula $\varphi(x,y) \in \mathcal{L}_{x,y}(\mathcal{U})$,  
\begin{equation*}
    \mu \otimes \nu (\varphi(x,y)) = \int_{S_{y}(N)} F_{\mu}^{\varphi} d(\nu|_{N}),
\end{equation*}
where $N$ is any small elementary submodel of $\mathcal{U}$ containing $M$ and any parameters from $\varphi$ and $\nu|_{N}$ is the associated regular Borel probability measure of the restriction of $\nu$ to $N$ on the type space $S_{y}(N)$.
\end{definition}

We remark that this this product is well-defined and the computation does not depend on our choice of $N$ (assuming $N$ contains $M$ and all parameters in $\varphi(x,y)$) (see discussion after \cite[Proposition 7.19]{Sibook}). This observation allows us to grow or shrink the space in which we are integrating over and we will make substantial use of this property in section 5. We end this section with a list of facts about measures and products. 

\begin{fact}\label{KM:imp2} Assume that $T$ is any theory and $\mu \in \mathfrak{M}_{x}(\mathcal{U})$, $\nu \in \mathfrak{M}_{y}(\mathcal{U})$, and $\lambda \in \mathfrak{M}_{z}(\mathcal{U})$. Assume that $\mu$ and $\nu$ are both $M$-invariant.
\begin{enumerate}[$(i)$]
\item If $\mu$ is smooth and $\nu$ is Borel definable, then $\mu \otimes \nu = \nu \otimes \mu$ $($see \cite[Corollary 2.5]{HPS}$)$. 
\item If $\mu$ and $\nu$ are definable (over $M$), then $\mu \otimes \nu$ is definable (over $M$) and $\mu \otimes (\nu \otimes \lambda) = (\mu \otimes \nu) \otimes \lambda)$ $($see \cite[Proposition 2.6]{CoGan}$)$.
\item If $\mu$ and $\nu$ are smooth (over $M$), then $\mu \otimes \nu$ is smooth (over $M$) $($e.g. \cite[Corollary 3.1]{CoGaNA}$)$.
\item If $\mu$ is Borel definable (over $M$) and $\nu$ is invariant (over $M$), then $\mu \otimes \nu$ is invariant (over $M$) $($discussion before \cite[Exercise 7.20]{Sibook}$)$.
\item If $\mu$ and $\nu$ are $\dfs$ (over $M$), then $\mu \otimes \nu$ is $\dfs$ (over $M$) $($e.g. \cite[Proposition 2.10]{CoGan}$)$.
\end{enumerate} 
Moreover, if $T$ is NIP then the following also hold. 
\begin{enumerate}[$(a)$]
\item If $\mu,\nu$ are invariant then $\mu \otimes (\nu \otimes \lambda) = (\mu \otimes \nu) \otimes \lambda$ $($see \cite{CoGaNA}$)$.
\item If $\mu$ is $\dfs$ and $\nu$ is invariant, then $\mu \otimes \nu = \nu \otimes \mu$ $($see \cite[Theorem 3.2]{HPS}$)$.
\end{enumerate} 
\end{fact} 

\begin{definition}[T is NIP]\label{prod:inf} Suppose that $\mu \in \mathfrak{M}_{x}(\mathcal{U})$ and $\mu$ is invariant. Then, we define the following measures: 
\begin{enumerate} 
\item $\mu^{(0)}(x_0) = \mu(x_0)$.
\item $\mu^{(n)} = \mu(x_{n}) \otimes \mu^{(n-1)}(x_0,...,x_{n-1})$. 
\item $\mu^{(\omega)} = \bigcup_{i \in \omega} \mu^{(n)}$ (where $\mu^{(\omega)} \in \mathfrak{M}_{(x_i)_{i \in \omega}}(\mathcal{U})$).  
\end{enumerate}
We note that $\mu^{(n)}$ and $\mu^{(\omega)}$ are well-defined by Fact \ref{KM:imp2}, and moreover we do not need to worry about the ordering of the parentheses in the product.
\end{definition}

\section{Sequentially approximated types and measures}
We begin this section by isolating the property of \textit{sequential approximability}. We again remark that these classes of objects are a global version of Khanaki's \textit{Baire 1 definability} \cite{Khanaki2}. We assume that $T$ is countable, but make no other global assumptions about $T$. As usual, $\mathcal{U}$ is a fixed sufficiently saturated model of $T$. We now define sequentially approximated types and measures.

\begin{definition}\label{SA} Let $p \in S_{x}(\mathcal{U})$ and  $\mu \in \mathfrak{M}_{x}(\mathcal{U})$. We say that, 
\begin{enumerate}
    \item $p$ is \textbf{sequentially approximated} if there exists $M \prec \mathcal{U}$ and a sequence of points $(a_i)_{i \in \omega}$ in $M^{x}$ such that $\lim_{i \to \infty} \tp(a_i/\mathcal{U}) = p$ in $S_{x}(\mathcal{U})$. In this case, we say $p$ is \textbf{sequentially approximated over $M$}. 
    \item $\mu$ is \textbf{sequentially approximated} if there exists $M \prec \mathcal{U}$ and a sequence of points $(\overline{a}_{i})_{i \in \omega}$ in $(M^{x})^{<\omega}$ such that $\lim_{i \to \infty} \Av(\overline{a}_{i}) = \mu$ in $\mathfrak{M}_{x}(\mathcal{U})$. In this case, we say $\mu$ is \textbf{sequentially approximated over $M$}.
\end{enumerate}
\end{definition}

We warn the reader that Definition \ref{SA} is only meaningful in the context of types and measures over large models. Indeed, if $M$ is a countable model and $T$ is a countable theory, then for every $p \in S_{x}(M)$, there exists a sequence of points in $M^{x}$ such that $\lim_{i \to \infty} \tp(a_{i}/M) =p$ in $S_{x}(M)$. The analogous statement also holds for measures.

We also emphasize to the reader that there is a real distinction between a type $p$ being sequentially approximated over a model $M$ and its associated Keisler measure $\delta_{p}$ being sequentially approximated over $M$. Proposition \ref{Gabe} gives an example of a type which is not sequentially approximated while its associated Keisler measure is sequentially approximated. However, the other implication holds almost trivially. 

\begin{observation}\label{forward:easy} If a type $p$ in $S_{x}(\mathcal{U})$ is sequentially approximated over a model $M$, then the associated Keisler measure $\delta_p$ is sequentially approximated over $M$.
\end{observation}  

\begin{proof} If $\lim_{i \to \infty}\tp(a_i/\mathcal{U}) = p$ in $S_{x}(\mathcal{U})$, then $\lim_{i \to \infty} \delta_{a_{i}} = \delta_{p}$ in $\mathfrak{M}_{x}(\mathcal{U})$ since $\delta: S_{x}(\mathcal{U}) \to \mathfrak{M}_{x}(\mathcal{U})$ is a topological embedding. 
\end{proof} 

\subsection{Basic properties} 

We now connect sequentially approximated types and measures to standard model-theoretic properties. For the reader's intuition, sequential approximability (at least in the case of measures) should be thought of as a strong version of finite satisfiability over a small model or a weak version of finite approximability. Sequentially approximated types remain a little more mysterious.  

\begin{proposition}\label{finitesat} Assume that $p \in S_{x}(\mathcal{U})$ and $\mu \in \mathfrak{M}_{x}(\mathcal{U})$. 
\begin{enumerate}[($i$)]
    \item If $p$ and $\mu$ are sequentially approximated over $M$, then $p$ and $\mu$ are finitely satisfiable over $M$. Even more, $p$ and $\mu$ are finitely satisfiable over a countable elementary submodel of $M$. 
    \item If $p$ and $\mu$ are sequentially approximated over $M$, then $p$ and $\mu$ are Borel-definable over $M$. 
    \item If $\mu$ is finitely approximated over $M$, then $\mu$ is sequentially approximated over $M$. $($Warning: In general, this fails for types.$)$
    \item If $T$ is NIP, then $p$ is sequentially approximated over $M$ if and only if $\delta_{p}$ is sequentially approximated over $M$. 
    \item Assume that $k \subseteq \{1,2,...,n\}$ and let $\pi_{k}:S_{n}(\mathcal{U}) \to S_{k}(\mathcal{U})$ and $\rho_{k}:\mathfrak{M}_{n}(\mathcal{U}) \to \mathfrak{M}_{k}(\mathcal{U})$ be the obvious projection maps. 
If $p \in S_{n}(\mathcal{U})$ and $p$ is sequentially approximated over $M$, then $\pi_{k}(p)$ is sequentially approximated over $M$. Similarly, if $\mu \in \mathfrak{M}_{n}(\mathcal{U})$ is sequentially approximated over $M$ then so is $\rho_{k}(\mu)$.
\end{enumerate}
\end{proposition}

\begin{proof} We prove the claims. 
\begin{enumerate}[($i$)] 
\item The first part of $(i)$ is obvious. For the second part, we only need to choose a submodel containing a sequence which sequentially approximates our type or measure. Since $T$ is countable, we can choose a countable model.
\item The proofs for both the type and measure cases are similar, so we prove the measure case. Assume that $(\overline{a}_{i})_{i \in \omega}$ is a sequence of points in $(M^{x})^{<\omega}$ such that $\lim_{i \to \infty} \Av(\overline{a}_{i}) = \mu$ in $\mathfrak{M}_{x}(\mathcal{U})$. By part $(i)$, $\mu$ is finitely satisfiable over $M$ and hence $M$-invariant. So, for any partitioned formula $\varphi(x,y)$ in $\mathcal{L}$, the map $F_{\mu}^{\varphi}:S_{y}(M)  \to [0,1]$ is well-defined. By sequential approximability, the sequence of continuous functions $\big(F_{\Av(\overline{a}_{i})}^{\varphi}\big)_{i \in \omega}$ converges pointwise to $F_{\mu}^{\varphi}$. Hence, $F_{\mu}^{\varphi}$ is Baire-1 (and therefore Borel). 
\item This follows from an encoding argument. Let $(\varphi_{n}(x,y_{n}))_{n \in \omega}$ be an enumeration of the partitioned $\mathcal{L}$-formulas. For each $n \in \mathbb{N}$, consider the partitioned formula $\theta_{n}(x;y_0,...,y_n,z_{*},z_0,...,z_n)$ where $|z_*| = |z_i| = 1$ and 
\begin{equation*} 
\theta_{n}(x;\bar{y},\bar{z}) : = \bigwedge_{i \leq n}\left( \left( z_* = z_i \wedge \bigwedge_{\substack{j \leq n \\ j \neq i }} z_{j} \neq z_{*} \right) \to \varphi_{i}(x,y_i) \right). 
\end{equation*} 
Since $\mu$ is finitely approximated over $M$, for $\epsilon = \frac{1}{n}$, there exists some $\overline{a}_n$ in $(M^{x})^{<\omega}$ such that for every $(\bar{b},\bar{c}) \in \mathcal{U}^{\bar{y}\bar{z}}$, 
\begin{equation*} |\Av(\overline{a}_n)(\theta_n(x,\bar{b},\bar{c})) - \mu((\theta_n(x,\bar{b},\bar{c})) | < \epsilon. 
\end{equation*} 
Notice that $\theta_{n}(x;\bar{y},\bar{z})$ encodes the definable sets which are obtained by the formulas $\varphi_{0}(x,y_0),...,\varphi_{n}(x,y_n)$. In particular, for every $b \in \mathcal{U}^{y_{j}}$ where $j \leq n$, consider then tuple $(\bar{d}_{b},\bar{c}_j) = (d_{0},...d_{j-1},b,d_{j+1}...,d_n,c_*,c_0,...,c_n)$ where the $d_{i}$'s are arbitrary and $c_* = c_l$ if and only if $l = j$. Then 
\begin{equation*}
 |\Av(\overline{a}_{n})(\varphi_{j}(x,b)) - \mu(\varphi_{j}(x,b))| = |\Av(\theta(x,\bar{d}_{b},\bar{c}_j)) - \mu((\theta(x,\bar{d}_{b},\bar{c}_j))|. 
\end{equation*}
So for any $j \leq n$ and $b \in \mathcal{U}^{y_{j}}$, 
\begin{equation*}
|\Av(\overline{a}_n)(\varphi_j(x,b)) - \mu(\varphi_j(x,b))| < \frac{1}{n}. 
\end{equation*}
It is clear that  $\lim_{n\to \infty} \Av(\overline{a}_{n}) = \mu$ in $\mathfrak{M}_{x}(\mathcal{U})$. 
\item The forward direction is Observation \ref{forward:easy}. We consider the converse. If $\delta_{p}$ is sequentially approximated over $M$ then $\delta_{p}$ is finitely satisfiable over a countable submodel $M_0$ by $(i)$ above. Then $p$ is finitely satisfiable over $M_0$ and so by Theorem \ref{sim:conv}, $p$ is sequentially approximated over $M_0$ (and also over $M$).
\item Simply consider the approximating sequence restricted to the appropriate coordinates. \qedhere
\end{enumerate} 
\end{proof} 

\begin{proposition}\label{Mazur} A measure $\mu$ is sequentially approximated and definable over $M$ if and only if $\mu$ is finitely approximated over $M$. 
\end{proposition}
\begin{proof} We first prove the forward direction. The proof is similar to the proof of \cite[Theorem 4.8]{GannNIP}. Fix $\epsilon > 0$. For any partitioned $\mathcal{L}$-formula $\varphi(x,y)$,  consider the map $F_{\mu}^{\varphi}:S_{y}(M) \to [0,1]$. Let $(\overline{a}_{i})_{i \in \omega}$ be a sequence of points in $(M^{x})^{<\omega}$ such that $\lim_{i \to \infty} \Av(\overline{a}_{i}) = \mu$ in $\mathfrak{M}_{x}(\mathcal{U})$. Observe that each map $F_{\Av(\overline{a})}^{\varphi}:S_{y}(M) \to [0,1]$ is continuous and the sequence $\big(F_{\Av(\overline{a}_{i})}^{\varphi}\big)_{i \in \omega}$ converge pointwise to $F_{\mu}^{\varphi}$. Since $\mu$ is definable, the map $F_{\mu}^{\varphi}$ is continuous. By the Riesz representation theorem and dominated convergence theorem, we have that $\big(F_{\Av(\overline{a}_{i})}^{\varphi}\big)_{i \in \omega}$ converges weakly to $F_{\mu}^{\varphi}$ in $C(S_{y}(M))$. By a standard application of Mazur's lemma, there exists a sequence of functions $(g_j)_{j \in \omega}$ such that each $g_j$ is a rational convex combination of $\{F_{\Av(\overline{a}_i)}^{\varphi}: i \leq n_{j}\}$ for some natural number $n_{j}$ and the sequence $(g_j)_{j \in \omega}$ converges uniformly to $F_{\mu}^{\varphi}$. Choose $m \in \mathbb{N}$ so that
\begin{equation*} 
\sup_{p \in S_{y}(M)}|F_{\mu}^{\varphi}(p) - g_m(p)| < \epsilon.
\end{equation*} 
By construction, $g_m = F_{\Av(\overline{c})}^{\varphi}$ for some $\overline{c} \in (M^{x})^{< \omega}$. Notice that 
\begin{equation*} \sup_{b \in \mathcal{U}^{y}}|\mu(\varphi(x,b)) - \Av(\overline{c})(\varphi(x,b))| < \epsilon.
\end{equation*}

For the converse, $\mu$ is definable over $M$ by $(iii)$ of Fact \ref{KM:imp}. Moreover, $\mu$ is sequentially approximated over $M$ by $(iii)$ of Proposition \ref{finitesat}.  
\end{proof} 

We now show that sequentially approximated measures commute with definable measures. It is well-known that in the context of NIP theories definable measures commute with measures which are finitely satisfiable over a small model (see  \cite[Lemma 3.1]{HPS} or \cite[Proposition 7.22]{Sibook}). Recently, it was shown that in general, measures which are finitely satisfiable over a small model (even $\dfs$ measures) do not always commute with definable measures (see \cite[Proposition 7.14]{CGH}). We first present a topological proof (in NIP theories) which shows that measures which are finitely satisfiable over a small model commute with definable measures. We will then modify this proof (by replacing an instance of continuity by the dominated convergence theorem) to show that sequentially approximated measures commute with definable ones in any theory. Recall the following facts.

\begin{fact}\label{cont:meas} Let $\nu \in \mathfrak{M}_{y}(\mathcal{U})$, $N \prec \mathcal{U}$, and $\varphi(x,y)$ be an $\mathcal{L}_{x,y}(N)$ formula. Let $\mathfrak{M}_{x}(\mathcal{U},N)$ denote the collection of measures in $\mathfrak{M}_{x}(\mathcal{U})$ which are finitely satisfiable over $N$. 
\begin{enumerate}[($i$)]
\item If $\nu$ is definable over $N$, then the map from $\mathfrak{M}_{x}(\mathcal{U})$ to $[0,1]$ defined via $\mu \to \nu \otimes \mu(\varphi(x,y))$ is continuous $($\cite[Lemma 5.4]{CGH}$)$. 
\item $($T is NIP$)$ If $\nu$ is any measure, then the map from $\mathfrak{M}_{x}(\mathcal{U},N)$ to $[0,1]$ defined via $\mu \to \mu \otimes \nu(\varphi(x,y))$ is well-defined and continuous $($\cite[Proposition 6.3]{ChGan}$)$.
\end{enumerate}
\end{fact}

We remark that statement $(ii)$ of Fact \ref{cont:meas} requires NIP for two reasons. First, it is not true in general that measures which are finitely satisfiable over a small model are Borel definable. In NIP theories, this is true ($(vi)$ of Fact \ref{KM:imp}). Secondly, the proof that this map is continuous relies on the existence of a smooth extension of $\nu|_N$. Without NIP, this map need not be continuous. The first proof of the following proposition can be found in \cite{HPS}.

\begin{proposition}[T is NIP] Assume that $\mu \in \mathfrak{M}_{x}(\mathcal{U})$ and $\nu \in \mathfrak{M}_{y}(\mathcal{U})$. If $\mu$ is finitely satisfiable over a small model and $\nu$ is definable, then $\mu \otimes \nu = \nu \otimes \mu$. 
\end{proposition} 

\begin{proof} Fix a formula $\varphi(x,y) \in \mathcal{L}_{x,y}(\mathcal{U})$. Choose $N$ such that $\mu$ is finitely satisfiable over $N$, $\nu$ is definable over $N$, and $N$ contains all the parameters from $\varphi$. Since $\mu$ is finitely satisfiable over $N$, there exists a net of measures $(\Av(\overline{a}_i))_{i \in I}$ such that each $\overline{a}_{i} \in (N^{x})^{< \omega}$ and $\lim_{i \in I} \Av(\overline{a}_i) = \mu$ in $\mathfrak{M}_{x}(\mathcal{U})$ ($(iv)$ of Fact \ref{Avcls}). By Fact \ref{cont:meas} 

\begin{align*} \mu \otimes \nu(\varphi(x,y)) = \int_{S_{y}(N)}F_{\mu}^{\varphi} d(\nu|_N) &\overset{(a)}{=}\ \lim_{i \in I} \int_{S_{y}(N)}F_{\Av(\overline{a}_i)}^{\varphi} d(\nu|_N)\\
& \overset{(b)}{=}\ \lim_{i \in I} \int_{S_{x}(N)}F_{\nu}^{\varphi^*} d(\Av(\overline{a}_{i})|_{N})\\
& \overset{(c)}{=}\ \int_{S_{x}(N)}F_{\nu}^{\varphi^*} d(\mu|_N) = \nu \otimes \mu (\varphi(x,y)).\\
\end{align*} 

Where the equalities $(a)$ and $(c)$ follow from the fact that continuous functions commute with nets. The equality $(b)$ is simple to check and is also justified by statement $(i)$ of Fact \ref{KM:imp2}.
\end{proof}

\begin{proposition}\label{prop:com} Sequentially approximated and definable measures commute. Assume that $\mu \in \mathfrak{M}_{x}(\mathcal{U})$ and $\nu \in \mathfrak{M}_{y}(\mathcal{U})$. If $\mu$ is sequentially approximated and $\nu$ is definable, then $\mu \otimes \nu = \nu \otimes \mu$. 
\end{proposition}
\begin{proof}
Fix a formula $\varphi(x,y) \in \mathcal{L}_{x,y}(\mathcal{U})$. Choose $N$ such that $\mu$ is sequentially approximated over $N$, $\nu$ is definable over $N$, and $N$ contains all the parameters from $\varphi$. Let $(\overline{a}_{i})_{i \in \omega}$ be a sequence of points in $(N^{x})^{<\omega}$ such that $\lim_{i \to \infty} \Av(\overline{a}_{i}) = \mu$ in $\mathfrak{M}_{x}(\mathcal{U})$. Now we consider the following computation. 

\begin{align*} \mu \otimes \nu(\varphi(x,y)) = \int_{S_{y}(N)} F_{\mu}^{\varphi} d(\nu|_N)  &\overset{(a)}{=}\ \lim_{i \to \infty}\int_{S_{y}(N)} F_{\Av(\overline{a}_{i})}^{\varphi} d(\nu|_N)\\
& \overset{(b)}{=}\ \lim_{i \to \infty} \int_{S_{x}(N)} F_{\nu}^{\varphi^{*}} d(\Av(\overline{a}_{i})|_N)\\
& \overset{(c)}{=}\ \int_{S_{x}(N)} F_{\nu}^{\varphi^{*}} d(\mu|_N) = \nu \otimes \mu(\varphi(x,y)).\\
\end{align*} 

Where the equality $(a)$ now holds from the dominated convergence theorem, equality $(c)$ holds from $(i)$ of Fact \ref{cont:meas} and the observation that continuous functions commute with nets, and equality $(b)$ is easy to check (also $(i)$ of Fact \ref{KM:imp2}). 
\end{proof} 

\begin{corollary} Let $\mu \in \mathfrak{M}_{x}(\mathcal{U})$ and $\nu \in \mathfrak{M}_{y}(\mathcal{U})$. If $\mu$ is finitely approximated and $\nu$ is definable, then $\mu \otimes \nu = \nu \otimes \mu$.
\end{corollary}

\begin{proof} By (iii) of Proposition \ref{finitesat}, $\mu$ is sequentially approximated. Apply Proposition \ref{prop:com}.
\end{proof} 

\subsection{Egorov's theorem} It is interesting to note that sequentially approximated measures are not too far away from finitely approximated measures. In particular, if we fix some measure on the parameter space, any sequentially approximated measure is \textit{almost} finitely approximated. This result is in a similar vein as Khanaki's \textit{almost definable} coheirs in the local setting (\cite{Khanaki1}). A direct application of Egorov's theorem gives our result. 

\begin{theorem}[Egorov's Theorem] Let $(X,B,\mu)$ be a finite measure space. Assume that $(f_i)_{i \in \omega}$ is a sequence of measurable functions from $X \to \mathbb{R}$ such that $(f_i)_{i \in \omega}$ converges to a function $f$ pointwise. Then for every $\epsilon > 0$ there exists a  $Y_{\epsilon} \in B$ such that $f_i|_{Y_{\epsilon}}$ converges to $f|_{Y_{\epsilon}}$ uniformly on $Y_{\epsilon}$ and $\mu(X \backslash Y_{\epsilon}) < \epsilon$. 
\end{theorem}

A proof of Egorov's theorem can be found in \cite[Theorem 3.2.4.1]{Acourse}. Restating this theorem in our context gives the following result.

\begin{corollary} Assume that $p$ and $\mu$ are sequentially approximated over $M$. Let $\nu \in \mathfrak{M}_{y}(M)$. Then, for every $\epsilon > 0$, there exists a Borel set $Y_{\epsilon} \subset S_{y}(M)$ such that
\begin{enumerate}
    \item $\nu(Y_{\epsilon}) > 1 - \epsilon$.
    \item For every $\delta > 0$ and every partitioned $\mathcal{L}$-formula $\varphi(x,y)$, there exists $\overline{a}_{\delta}$ in $(M^{x})^{<\omega}$ such that for every $b \in \mathcal{U}^{y}$ so that $\tp(b/M) \in Y_{\epsilon}$, we have
    \begin{equation*}
        |\mu(\varphi(x,b)) - \Av(\overline{a}_{\delta})(\varphi(x,b))| < \delta.
    \end{equation*}
    \item For every partitioned $\mathcal{L}$-formula $\varphi(x,y)$, there exists $a$ in $M^{x}$ such that for every $b \in \mathcal{U}^{y}$ so that $\tp(b/M) \in Y_{\epsilon}$, we have
    \begin{equation*}
        \varphi(x,b) \in p \iff \models \varphi(a,b).
    \end{equation*}
\end{enumerate}
\end{corollary} 

\section{Generically stable types}

Throughout this section, we let $T$ be a countable theory and $\mathcal{U}$ be a monster model of $T$. We show that if a type $p$ is generically stable over a small submodel $M$ of $\mathcal{U}$, then $p$ is sequentially approximated over $M$. Toward proving this result, we actually prove a slightly stronger lemma than what is necessary. Namely, let $p$ be a $\dfs$ type and let $M$ be a countable model such that $p$ is $\dfs$ over $M$ (for any $\dfs$ type, these models always exist by (iv) of Fact \ref{gfs:facts}). We show that there exists a special sequence of points in $M$ such that the \textit{limiting behavior} of this sequence \textit{resembles} a Morley sequence in $p$ over $M$.  In the case where $p$ is generically stable over $M$, we show that this special sequence converges to $p$. This is enough to show the result since every generically stable type is generically stable over some countable model. We now begin with a discussion on eventually indiscernible sequences, which were introduced in \cite{Invariant}. 

\begin{definition} Let $(c_i)_{i \in \omega}$ be a sequence of points in $\mathcal{U}^{x}$ and $A \subset \mathcal{U}$. We say that $(c_i)_{i \in \omega}$ is an \textbf{eventually indiscernible sequence over $A$} if for any formula $\varphi(x_0,...,x_k)$ in $\mathcal{L}_{(x_i)_{i \in \omega}}(A)$, there exists some natural number $N_{\varphi}$ such that for any indices $n_{k} > .... > n_{0} > N_{\varphi}$ and $m_{k} > ... > m_{0} > N_{\varphi}$, we have that 
\begin{equation*}
\mathcal{U} \models \varphi(c_{n_0},...,c_{n_{k}}) \leftrightarrow \varphi(c_{m_0},...,c_{m_k}).
\end{equation*} 
\end{definition} 

\begin{fact}\label{eventual} Let $(b_i)_{i \in \omega}$ be a sequence of points in $\mathcal{U}^{x}$ and $A \subset \mathcal{U}$ such that $|A| = \aleph_0$. Then there exists a subsequence $(c_i)_{i \in \omega}$ of $(b_i)_{i \in \omega}$ such that $(c_i)_{i \in \omega}$ is eventually indiscernible over $A$.
\end{fact} 
The proof is a standard application of Ramsey's theorem and taking the diagonal (as mentioned in \cite{Invariant}). We prove a ``continuous" version of this fact in the next section and the proof is analogous (see Proposition \ref{correct} for details). 

For any eventually indiscernible sequence $(c_{i})_{i \in \omega}$ over a set of parameters $A$, we can associate to this sequence a unique type in $S_{(x_i)_{i \in \omega}}(A)$. We call this the \textit{eventual  Ehrenfeucht-Mostowski type} (or $\EEM$-type) of $(c_{i})_{i \in \omega}$ over $A$. We now give the formal definition.

\begin{definition} Let $(b_{i})_{i \in \omega}$ be a sequence of points in $\mathcal{U}^{x}$ and $A \subset \mathcal{U}$. Then the \textbf{eventual Ehrenfeucht-Mostowski type} (or \textbf{EEM-type}) of $(b_i)_{i \in \omega}$ over $A$, which is written as $\EEM((b_i)_{i \in \omega})/A)$, is a subset of $\mathcal{L}_{(x_i)_{i \in \omega}}(A)$ defined as follows:

Let $\varphi(x_{i_{0}},...,x_{i_{k}})$ be a formula in $\mathcal{L}_{(x_{i})_{i \in \omega}}(A)$ where the indices are ordered $i_{0} < ... < i_{k}$. Then $\varphi(x_{i_0},...x_{i_{k}}) \in \EEM((b_{i})_{i \in \omega})/A)$ if and only if there exists an $N_{\varphi}$ such that for any $n_k > ... > n_0 > N_{\varphi}$, we have that $\mathcal{U} \models \varphi(b_{n_0},..., b_{n_k})$.
\end{definition}

Notice that an $\EEM$-type of a sequence is always indiscernible in the following sense: If we have indices $i_{0},...,i_{k}$ and $j_{0},...,j_{k}$ where $i_{0} < ... < i_{k}$ and $j_{0}<...<j_{k}$, 
then $\varphi(x_{i_{0}},...,x_{i_{k}})$ is in the 
$\EEM$-type of $(b_{i})_{i \in \omega}$ over $A$ if and only if $\varphi(x_{j_0},...,x_{j_k})$ is. This follows directly from the definition. We have some basic observations.

\begin{observation} Let be $(c_i)_{i \in \omega}$ an eventually indiscernible sequence over $A$. 
\begin{enumerate}
\item Then $\EEM((c_{i})_{i \in \omega}/A)$ is a complete type in $S_{(x_i)_{i \in \omega}}(A)$.
\item If $(c_i)_{i \in \omega}$ is $A$-indiscernible, then $\EEM((c_i)_{i \in \omega}/A) = \EM((c_i)_{i \in \omega}/A)$. 
\item If $\tp((b_i)_{i \in \omega}/A) = \EEM((c_i)_{i \in \omega}/A)$, then $(b_i)_{i \in \omega}$ is $A$-indiscernible. 
\end{enumerate} 
\end{observation} 

\begin{proof} Clear from the definitions and discussion above. 
\end{proof}

We warn the reader that an eventually indiscernible sequence need not ``realize" its own $\EEM$-type. Consider the following example: 

\begin{example} Let $T_{<}$ be the theory of $(\mathbb{R};<)$. Let $\mathcal{U}$ be a monster model of $T_{real}$ and $\mathbb{R} \prec \mathcal{U}$. Then the sequence $(a_i)_{i \in \omega}$ where $a_i = i$ is eventually indiscernible over $\mathbb{R}$ while the sequence $(b_i)_{i \in \omega}$ where $b_i = i(-1)^{i}$ is not. Clearly, $(a_i)_{i \in \omega}$ is not $\mathbb{R}$-indiscernible. Moreover, for each $r \in \mathbb{R}$, the formula $x_0 > r$ is in $\EEM((a_i)_{i \in \omega}/\mathbb{R})$ while $a_{1} >2$ clearly does not hold. So if $\tp((c_i)_{i \in \omega}/\mathbb{R})) = \EEM((a_i)_{i \in \omega}/\mathbb{R})$, then $c_i > \mathbb{R}$ for each $i \in \omega$. 
\end{example} 

The next two lemmas prove the bulk of this section's main theorem and their proofs are similar to the proof of Theorem \ref{sim:conv}. The proof strategy for this theorem is the following: If $p$ is in $S_{x}(\mathcal{U})$ and $p$ is $\dfs$, then we can find a countable model $M$ such that $p$ is $\dfs$ over $M$. Let $I$ be a Morley sequence in $p$ over $M$. Using the fact that $p$ is finitely satisfiable over $M$, we can find a sequence of points in $M^{x}$ which converge to $p|_{MI}$ in $S_{x}(MI)$. After moving to an eventually indiscernible subsequence, we show that the $\EEM$-type of this eventually indiscernible sequence is precisely $p^{\omega}|_{M}$. With the stronger assumption that our type $p$ is generically stable (instead of just $\dfs$), we show that this eventually indiscernible subsequence must converge to $p$ in $S_{x}(\mathcal{U})$.

\begin{lemma}\label{dfs:lemma} Suppose $p$ is in $S_{x}(\mathcal{U})$ and $p$ is $\dfs$ over $M$ where $|M| = \aleph_0$. Then there exists a sequence $(c_i)_{i \in \omega}$ in $M^{x}$ such that $\EEM((c_i)_{i \in \omega}/M) = p^{\omega}|_{M}$. 
\end{lemma} 

\begin{proof} Let $I = (a_i)_{i \in \omega}$ be a Morley sequence in $p$ over $M$. Since $T$, $M$, and $I$ are countable, $\mathcal{L}_{x}(MI)$ is countable. It follows that $p|_{MI}$ is countable and we may enumerate this collection of formulas as $(\varphi_{i}(x))_{i \in \omega}$. Since $p$ is $\dfs$ over $M$, in particular $p$ is finitely satisfiable over $M$. For each natural number $n$, we choose $b_{n}$ in $M^{x}$ such that $\mathcal{U} \models \bigwedge_{j \leq n} \varphi_{j}(b_n)$. By construction, we have that $\lim_{i \to \infty} \tp(b_i/MI) = p|_{MI}$ in $S_{x}(MI)$. By Fact \ref{eventual}, we may choose a subsequence $(c_{i})_{i \in \omega}$ of $(b_{i})_{i \in \omega}$ such that $(c_i)_{i \in \omega}$ is eventually indiscernible over $MI$. For ease of notation, we write $(c_{i})_{i \in \omega}$ as $J$.

We now show that \textit{$\EEM(J/M) = \EM(I/M) =  p^{\omega}|_{M}$}. We remind the reader that $\EM(I/M) = p^{\omega}|_{M}$ follows directly from the definition of a Morley sequence. We prove the first equality by induction on the number of free variables occurring in a formula. We begin with the base case. It suffices to show that for every $\varphi(x_0) \in \mathcal{L}_{x_{0}}(M)$, if $\varphi(x_0) \in \EM(I/M)$, then $\varphi(x_0) \in \EEM(J/M)$.  Notice that  $ \lim _{n\to \infty} \tp(b_n/MI) = p|_{MI}$, and $(c_i)_{i \in \omega}$ is a subsequence of $(b_n)_{n \in \omega}$, $\lim _{i\to \infty} \tp(c_i/MI) = p|_{MI}$. This clearly implies the base case. 

 Fix $k$ and suppose that for any formula $\theta(x_0,...,x_k)$ in $\mathcal{L}_{x_{0},...,x_{k}}(M)$, we have that $\theta(x_0,...,x_k) \in \EM(I/M)$ if and only if $\theta(x_0,...,x_k) \in \EEM(J/M)$.

Towards a contradiction, we assume that $\neg \theta(x_0,...,x_{k+1}) \in \EEM(J/M)$ and $\theta(x_0,...,x_{k+1}) \in \EM(I/M)$. Since $\neg \theta(\overline{x}) \in \EEM(J/M)$, there exists some natural number $N_{\theta_{1}}$ such that for any $n_{k+1} > ... > n_{0} > N_{\theta_{1}}$, we have that $\mathcal{U}\models \neg \theta(c_{n_{0}},...,c_{n_{k+1}})$. Since $\theta(\overline{x}) \in \EM(I/M)$, we conclude that $\mathcal{U} \models \theta(a_0,...,a_{k+1})$. Since $p$ is $\dfs$ over $M$, $I$ is totally indiscernible over $M$ by Fact \ref{gfs:facts}. Therefore, $\mathcal{U} \models \theta(a_{k+1}, a_{0}...,a_{k})$ and so $\theta(x,a_{0},...,a_{k}) \in p|_{Ma_{0},...,a_{k}}$. Since $\lim_{i \to \infty}\tp(c_i/MI) = p|_{MI}$, there exists some $N_{\theta_2}$ such that for every $n > N_{\theta_{2}}$, we have that $\mathcal{U}\models \theta(c_{n},a_{0},...,a_{k})$. Choose $n_{*} > \max\{N_{\theta_{1}},N_{\theta_{2}}\}$. Then the formula $\theta(c_{n_*},x_{0},...,x_{k}) \in \tp(a_0,...,a_{k}/M)$. By our induction hypothesis, we have that $\theta(c_{n_{*}},\overline{x}) \in \EEM(J/M)$ and so there exists $N_{\theta_{3}}$ such that for any $m_{k}> ... > m_{0} > N_{\theta_{3}}$, we have that $\mathcal{U}\models \theta(c_{n_*}, c_{m_{0}},...,c_{m_{k}})$. Now consider what happens when $m_0 > \max\{N_{\theta_{3}}, n_{*}\}$. Then $m_k > ... > m_{0} > n_{*} > N_{\theta_1}$ and so $\mathcal{U} \models \neg \theta(c_{n_*},c_{m_0},...,c_{m_k})$ by our assumption. However, $m_k > ... > m_{0} > N_{\theta_{3}}$ and therefore $\mathcal{U} \models \theta(c_{n_*},c_{m_0},...,c_{m_{k}})$. This is a contradiction.
\end{proof}

\begin{lemma}\label{gs:lemma} Suppose $p$ is in $S_{x}(\mathcal{U})$ and $M \prec \mathcal{U}$. Assume that $p$ is generically stable over $M$. If $(c_i)_{i \in \omega}$ is a sequence in $M^{x}$ such that $\EEM((c_i)_{i \in \omega}/M) = p^{\omega}|_{M}$, then $\lim_{i \to \infty} tp(c_i/\mathcal{U}) = p$. 
\end{lemma}
\begin{proof} Let $p$, $(c_{i})_{i \in \omega}$ and $M$ be as in the statement of the lemma. Let $J = (c_{i})_{i \in \omega}$. We first argue that the sequence of global types $(\tp(c_i/\mathcal{U}))_{i \in \omega}$ converges and then argue that this sequence converges to $p$. 

\textbf{Claim 1:} The sequence $(\tp(c_i/\mathcal{U}))_{i \in \omega}$ converges to a some type in $S_{x}(\mathcal{U})$.

It suffices to argue that for any formula $\psi(x) \in \mathcal{L}_{x}(\mathcal{U})$, $\lim_{i \to \infty} \mathbf{1}_{\psi}(c_i)$ exists (recall that $\mathbf{1}_{\psi(x)}$ is the characteristic function of the definable set $\psi(x)$). Assume not. Then we may choose a subsequence $(c_i')_{i \in \omega}$ of $(c_i)_{i \in \omega}$ such that $ \mathcal{U} \models \psi(c_{i}') \leftrightarrow \neg \psi(c_{i+1}')$. For notational purposes, we also denote $(c'_i)_{i \in \omega}$ as $J'$. It is clear that $(c_{i}')_{i \in \omega}$ is also eventually indiscernible over $M$ and $\EEM((c_i')_{i \in \omega}/M) = \EEM((c_i)_{i \in \omega}/M)$. By using $J'$, one can show that the following type is finitely consistent:  
\begin{equation*}
    \Theta_1 = \EEM(J'/M) \cup  \bigcup_{\textit{$i$ is even}}\{\psi(x_i) \wedge \neg \psi(x_{i+1})\}. 
\end{equation*}
Let $(d_i)_{i \in \omega}$ realize this type. Then $(d_i)_{i \in \omega}$ is a Morley sequence in $p$ over $M$ because
\begin{equation*}
   \EM((d_i)_{i \in \omega}/M) = \EEM(J'/M) = \EEM(J/M) = p^{\omega}|_{M}.
\end{equation*} Then $\mathcal{U} \models \psi(d_{i})$ if and only if $i$ is even. This contradicts generic stability since $\lim_{i \to \infty} \tp(d_i/M)$ does not converge.  

\textbf{Claim 2:} The sequence $(\tp(c_i/\mathcal{U}))_{i \in \omega}$ converges to $p$. 

Again, assume not. By claim 1, $\lim_{i \to \infty} \tp(c_i/\mathcal{U}) = q$ for some $q \in S_{x}(\mathcal{U})$. By assumption, $q \neq p$ and so there exists a formula $\psi(x)$ such that $\psi(x) \in p$ and $\neg \psi(x) \in q$. Since $(\tp(c_i/\mathcal{U}))_{i \in \omega}$ converges to $q$, there is an $N$ such that for every $n > N$, we have that $\mathcal{U} \models \neg \theta(c_n)$. By a similar argument as the previous claim, one can show the following type is finitely consistent: 
\begin{equation*}
    \Theta_2 = \EEM(J/M) \cup \bigcup_{i \in \omega} \neg \theta(x_{i}). 
\end{equation*}
Again, we let $(d_i)_{i \in \omega}$ realize this type. Then $(d_i)_{i \in \omega}$ is a Morley sequence in $p$ over $M$ and we have that $\lim_{i \to \infty} \tp(d_i/\mathcal{U}) \neq p$ in $S_{x}(\mathcal{U})$. This again contradicts the definition of generic stability.
\end{proof}

\begin{theorem}\label{gstheorem} Suppose $p$ is in $S_{x}(\mathcal{U})$ and $p$ is generically stable (over $M$). Then $p$ is sequentially approximated (over $M$).
\end{theorem}

\begin{proof} If $p$ is generically stable, then $p$ is generically stable over a countable submodel model $M_{0}$ contained in $M$ by Fact \ref{gfs:facts}. Then $p$ is $\dfs$ over $M_0$ and so by Lemma \ref{dfs:lemma}, one can choose $(c_i)_{i \in \omega}$ where each $c_i \in M_{0}^{x}$ and $\EEM((c_i)_{i \in \omega}/M_{0}) = p^{\omega}|_{M_0}$. By Lemma \ref{gs:lemma}, $\lim_{i \to \infty}\tp(c_i/\mathcal{U}) = p$.
\end{proof}

\begin{corollary} Assume that $T'$ is countable or uncountable in the language $\mathcal{L'}$, $\mathcal{U}' \models T'$, and $M'$ a submodel of $\mathcal{U}'$. Assume that $p$ is generically stable over $M'$. Then for any countable collection of formulas $\Delta = \{\psi_{i}(x,y_{i})\}_{i \in \omega}$ in $\mathcal{L'}$, there exists a sequence of points $(c_i)_{i \in \omega}$ each in $(M')^{x}$ such that $\lim_{i \to \infty} \tp_{\Delta}(c_i/\mathcal{U}) = p|_{\Delta}$.
\end{corollary} 

\begin{proof} Let $\mathcal{L}$ be a countable sublanguage of $\mathcal{L'}$ containing all the formulas in $\Delta$. The corresponding type $p|_{\mathcal{L}}$ is generically stable over the model $M$ where $M = M'|_{\mathcal{L}}$ (see \cite[Remark 3.3]{CoGan}). Hence we may apply Theorem \ref{gstheorem}.
\end{proof} 

\subsection{Examples and non-examples}

We begin this subsection by collecting the known examples of sequentially approximated types. We then go on to give two examples of types which are not sequentially approximated (over any model).

\begin{observation} Assume that $p \in S_{x}(\mathcal{U})$ and let $M$ be a small elementary submodel. Then, $p$ is sequentially approximated over $M$ if

\begin{enumerate}[($i$)]
\item $T$ is stable, and $p$ is invariant over $M$,
\item $T$ is NIP, $|M| = \aleph_0$, and $p$ is finitely satisfiable over $M$, or
\item $p$ is generically stable over $M$.
\end{enumerate}
\end{observation}

We just proved $(iii)$. Clearly, $(i)$ follows from $(iii)$ (we remark that it also follows from $(ii)$). As noted previously, the proof of $(ii)$ is precisely \cite[Lemma 2.8]{Invariant}. 

We now exhibit some concrete examples of types which are not sequentially approximated. We begin by describing a type in an NIP theory which is finitely satisfiable over a small model but not sequentially approximated (and its associated Keisler measure is not sequentially approximated either). We then discuss a finitely approximated type which is not sequentially approximated.

\begin{proposition}\label{omega} Let $\omega_{1}$ be the first uncountable ordinal, $M = (\omega_{1};<)$ with the usual ordering, and let $T_{<}$ be the theory of $M$ in the language $\{<\}$.  Recall that $T_{<}$ is NIP. Let $p \in S_{x}(\omega_{1})$ be the complete type extending $\{\alpha < x: \alpha < \omega_1\}$. Let $\mathcal{U}$ be a monster model of $T_{<}$ such that $M \prec \mathcal{U}$ and let $p_* \in S_{x}(\mathcal{U})$ be the unique global coheir of $p$. Then, $p_{*}$ is not sequentially approximated over any model. 
\end{proposition}

\begin{proof} Assume for the sake of contradiction that $p_{*}$ is sequentially approximated over some model $N$. Then there exists a sequence of points $(b_i)_{i \in \omega}$ in $N$ such that $\lim_{i \to \infty} \tp(b_i/\mathcal{U}) =  p_{*}$ in $S_{x}(\mathcal{U})$. There is either an infinite subsequence which is strictly increasing or strictly decreasing and so without loss of generality, $(b_i)_{i \in \omega}$ has one of these two properties. First assume that  $(b_i)_{i \in \omega}$ is strictly increasing. Notice that $b_i < x \in p_{*}$. Since $p_{*}$ is a coheir of $p$, $p_*$ is finitely satisfiable over $\omega_1$. So, for each $b_i$ there exists $\alpha$ in $\omega_{1}$ such that $b_{i} < \alpha$. Now, for each $b_i$, we define  
$\alpha_i := \min \{\alpha \in \omega_1: \mathcal{U} \models b_i<\alpha \}$. Since $\omega_1$ is well-ordered, $\alpha_i$ is well-defined. We let $\beta$ be the supremum (in $\omega_{1}$) of $\{\alpha_{i}: i \in \omega\}$. Then $\mathcal{U} \models b_i < \beta$ for each $i \in \omega$, and so  but $x < \beta \in p_{*}$, contradiction. 

Now we assume that $(b_i)_{i \in \omega}$ is a strictly decreasing subsequence. Notice that for each $i \in \omega$, $b_i>x \in p_{*}$. Let $\Theta(x) = \{\alpha < x: \alpha \in \omega_1\}\ \cup \{x < b_i: i \in \omega\}$. By compactness, choose $c_{\infty}$ in $\mathcal{U}$ satisfying $\Theta(x)$. Since $p_{*}$ is finitely satisfiable over $\omega_1$, we have $c_{\infty} > x \in p_{*}$. But since $\mathcal{U} \models b_i > c_{\infty}$ for each $i \in \omega$, we have that $x >c_{\infty} \in p$, contradiction. 
\end{proof}

\begin{remark}\label{example:coheir1} The type $p_{*}$ in Proposition \ref{omega} is finitely satisfiable over a small model, but not finitely satisfiable over any countable submodel by Theorem \ref{sim:conv}. 
\end{remark} 

\begin{proposition} Let $p_{*}$ be as in Proposition \ref{omega}. Then the associated Keisler measure $\delta_{p_{*}}$ is not sequentially approximated. 
\end{proposition}
\begin{proof} Clear from $(iv)$ of Proposition \ref{finitesat}. 
\end{proof} 

\begin{proposition}\label{Gabe} Let $T^{2}_{s}$ be the theory of the random $K_{s}$-free graph in the language $\mathcal{L} = \{E(x,y)\}$. Let $p_{*}$ be the unique global complete type extending the formulas $\{ \neg E(x,b): b \in \mathcal{U}\}$. Then, $\delta_{p_{*}}$ is sequentially approximated (even finitely approximated over any submodel) but $p_{*}$ is not sequentially approximated. Moreover, $T^{2}_{s}$ admits no (non-realized) sequentially approximated types.
\end{proposition} 

\begin{proof} The proof that $\delta_{p_*}$ is finitely approximated can be found in \cite[Theorem 5.8]{CoGan}. By $(iii)$ of Proposition \ref{finitesat}, $\delta_{p_*}$ is sequentially approximated. By $(v)$ of Proposition \ref{finitesat}, it suffices to show that there are no non-realized types in one variable which are sequentially approximated. Let $p$ be any non-realized type in $S_{1}(\mathcal{U})$ and assume that $(b_i)_{i \in \omega}$ is a sequence of points in $\mathcal{U}^{x}$ such that $\lim_{i\to \infty}\tp(b_i/\mathcal{U}) = p$. Since $p$ is non-realized, we may assume that the points in $(b_i)_{i \in \omega}$ are distinct.
Then, by Ramsey's theorem, there is a subsequence which is either independent or complete. It cannot be complete, because that would violate $K_{s}$-freeness. Therefore, $(b_i)_{i \in \omega}$ contains an independent subsequence, call it $(c_i)_{i \in \omega}$. By compactness, there exists an $a$ in $\mathcal{U}$ such that $\mathcal{U} \models E(c_i,a)$ if and only if $i$ is even. Then, $(\tp(c_i/\mathcal{U}))_{i \in \omega}$ does not converge in $S_{x}(\mathcal{U})$ and so $(\tp(b_i/\mathcal{U}))_{i \in \omega}$ does not converge in $S_{x}(\mathcal{U})$.
\end{proof}

\begin{question} We say a global type $p$ in $S_{x}(\mathcal{U})$ is \textbf{sad}\footnote{Credit to James Hanson for the terminology.} if it is both \textbf{s}equentially \textbf{a}pproximated and \textbf{d}efinable. Does there exist a global type $p$ which is sad over a model $M$ but is not generically stable over $M$? It is clear that if $T$ is NIP, then all sad types are generically stable. Therefore an example of such a type must come from \textit{the wild}.
\end{question} 

\section{Sequential approximations of measures in NIP theories} 
Throughout this section, we assume that $T$ is a countable NIP theory and $\mathcal{U}$ is a monster model of $T$. We show that measures which are finitely satisfiable over a countable model of $T$ are sequentially approximated (Theorem T2). To do this, we introduce the notion of a \textit{smooth sequence}. These are sequences of global measures which are intended to play the role of a Morley sequence for a measure. Unfortunately, these sequences only exist (a priori) in the NIP context and it is currently not known how to expand this idea to IP theories. At the end of this section, we give a characterization of generic stability using smooth sequences (again, only in the NIP context). 

To motivate the machinery introduced in this section, we explain why Theorem T2 does not follow directly from some approximation results currently in the literature. One might assume that one could prove Theorem T2 from Theorem \ref{sim:conv} in tandem with the following fact \cite[Proposition 7.11]{Sibook},

\begin{fact}[T is NIP]\label{type:approx} Suppose that $\mu \in \mathfrak{M}_{x}(\mathcal{U})$ and $\mu$ is finitely satisfiable over $M$. Then, for any formula $\varphi(x,y) \in \mathcal{L}$ and every $\epsilon > 0$, there exists types $p_1,...,p_n \in S_{x}(\mathcal{U})$, where for each $i \leq n$ the type $p_i$ is finitely satisfiable over $M$, and 
\begin{equation*}
\sup_{b \in \mathcal{U}^{y}}| \mu(\varphi(x,b)) - \Av(\overline{p})(\varphi(x,b))|< \epsilon. 
\end{equation*} 
\end{fact}

If $\mu$ is in $\mathfrak{M}_{x}(\mathcal{U})$ and is finitely satisfiable over a countable model $M$, then one can use Theorem \ref{sim:conv} and Fact \ref{type:approx} together to produce: 
\begin{enumerate}
\item  a sequence of global measures $(\Av(\overline{p}_{i}))_{i \in \mathbb{N}}$ such that each $\overline{p}_{i} = (p_{i_1},...,p_{i_{k}})$, each $p_{i_{k}} \in S_{x}(\mathcal{U})$ is finitely satisfiable over $M$, and $\lim_{i \to \infty} \Av(\overline{p}_i) = \mu$ in $\mathfrak{M}_{x}(\mathcal{U})$,
\item for each $i \in \mathbb{N}$, a sequence of points $(\overline{a}_{i_j})_{j \in \mathbb{N}}$ each in $(M^{x})^{<\omega}$ so that $\lim_{j \to \infty} \Av(\overline{a}_{i_j})= \Av(\overline{p}_{i})$.
\end{enumerate}
This construction gives an \textit{array} of points $(\overline{a}_{i_j})_{(i,j) \in \mathbb{N} \times \mathbb{N}}$ in $(M^{x})^{< \omega}$ so that 
\begin{equation*}\lim_{i \to \infty} \lim_{j \to \infty} \Big(\Av(\overline{a}_{i_j})\Big) = \mu \text{ in $\mathfrak{M}_{x}(\mathcal{U})$}.
\end{equation*}
A priori, the convergence of an array \textit{does not imply} that there exists a subsequence of that array which converges to the array's limit\footnote{For example, any Baire-2 function which is not Baire-1 can be written as the limit of an array of continuous functions, but cannot be written as the sequential limit of continuous functions.}. A similar situation arises by trying to iterate Theorem \ref{Khanaki}. So, we must work slightly harder. As previously stated, our proof essentially mimics the proof of Theorem \ref{sim:conv} but with Morley sequences replaced by \textit{smooth sequences}. Finally we remark that if there were an \textit{elementary proof} using an array to show this result, then we would have a moderately simple proof that $\dfs$ measures are finitely approximated in NIP theories. In particular, this proof would bypass the implicit use of randomizations (i.e. $(i)$ of Fact \ref{HPSFact}). 

We formally begin this section by discussing a ``continuous" analogue of eventually indiscernible sequences.

\subsection{Eventually indiscernible sequences revisited}  

We fix some notation. Fix distinct tuples of variables $x$ and $x_0,...,x_n$ such that $|x| = |x_i|$ for $i \leq n$. If $\varphi(x_0,...,x_n)$ is a formula in $\mathcal{L}_{x_0,...,x_n}(\mathcal{U})$ and $\overline{a}_{0},...,\overline{a}_{n}$ is a finite sequence of elements where each $\overline{a}_i \in (\mathcal{U}^{x})^{<\omega}$ and $\overline{a}_{i} = (a_{i,0},...,a_{i,m_{i}})$ for $i \leq n$, then we write $\varphi_{c}(\overline{a}_{0},...,\overline{a}_{n})$ to mean,
\begin{equation*}
\bigotimes_{i=0}^{n}\Av(\overline{a}_i)_{x_{i}}(\varphi(x_0,...,x_n)).
\end{equation*} 
Notice that $\varphi_{c}(\overline{a}_0,...,\overline{a}_n)$ is a real number. We observe that by unpacking the definition of the product measure, our formula can be computed as follows: 
\begin{equation*}
\varphi_{c}(\overline{a}_{0},...,\overline{a}_{n})= \frac{1}{\prod_{i=0}^{n} (m_{i} + 1)} \sum_{j_{0} = 0}^{m_{0}}...\sum_{j_{n} = 0}^{m_{n}}\mathbf{1}_{\varphi}(a_{0,j_0},...,a_{n,j_{n}}).
\end{equation*} 

\begin{definition}\label{convex} Let $(\overline{a}_{i})_{i \in \omega}$ be a sequence of elements in $(\mathcal{U}^{x})^{<\omega}$ and let $A \subset \mathcal{U}$ be a collection of parameters. Then we say that the sequence $(\overline{a}_i)_{i \in \omega}$ is \textbf{eventually indiscernible over $A$} if for any formula $\varphi(x_0,...,x_n)$ in $\mathcal{L}_{(x_i)_{i\in \omega}}(A)$ and any $\epsilon > 0$, there exists $N_{\epsilon,\varphi}$ such that for any $n_{k}>...>n_{0}>N_{\epsilon,\varphi}$ and $m_{k}>....>m_{0}>N_{\epsilon,\varphi}$, 
\begin{equation*}
|\varphi_{c}(\overline{a}_{n_{0}},...,\overline{a}_{n_{k}})-\varphi_{c}(\overline{a}_{m_{0}},...,\overline{a}_{m_{k}})|<\epsilon.
\end{equation*}

\end{definition} 

\begin{proposition}\label{correct} Let $(\overline{a}_{i})_{i\in\omega}$ be a sequence of tuples in $(\mathcal{U}^{x})^{< \omega}$. If $A$ is a countable set of parameters, then there exists some subsequence $(\overline{c}_i)_{i \in \omega}$ of $(\overline{a})_{i \in \omega}$ such that $(\overline{c}_{i})_{i \in \omega}$ is eventually indiscernible over $A$. 
\end{proposition}
\begin{proof} This proof is a standard application of Ramsey's theorem applied to the ``continuous" setting. Enumerate all pairs in $\mathcal{L}_{(x_i)_{i \in \omega}}(A) \times \mathbb{N}_{>0}$. Let $(\overline{a}_{i}^{0})_{i\in\omega} :=(\overline{a}_{i})_{i\in\omega}$ and set $B_{0} = \{\overline{a}^{0}_{i}: i \in \omega\}$. Now, assume we have constructed the subsequence $(\overline{a}_{i}^{l})_{i\in\omega}$ and $B_{l}$ (where $B_{l} = \{\overline{a}_{i}^{l}: i \in \omega\}$). We now construct $(\overline{a}_{i}^{l+1})_{i\in\omega}$ and $B_{l+1}$. Assume that $(\varphi(x_{0},...,x_{k}),n)$ is the $l+1$ indexed pair in our enumeration. Then we define the coloring $r_{l+1}:[B_{l}]^{k+1}\to\{0,...,n\}$ via 
\begin{equation*}
r(\{\overline{a}^{l}_{i_{0}},...,\overline{a}^{l}_{i_{k}}\}) = \lfloor n \cdot \varphi_{c}(\overline{a}^{l}_{i_{0}},...,\overline{a}^{l}_{i_{k}})   \rfloor.
\end{equation*}
where $i_0 < i_1 < ...< i_k$. 
By Ramsey's theorem, there is an infinite monochromatic subset $B_{l}'$ of $B_{l}$. Let $(\overline{a}_{i}^{l+1})_{i\in\omega}$ be the obvious reindexed subsequence of $(\overline{a}_{i}^{l})_{i\in\omega}$ with the elements only from the monochromatic set $B_{l}^{'}$. We let $B_{l + 1} = \{\overline{a}^{l+1}_{i}: i \in \omega\}$. By construction, the sequence $(\overline{a}_{i}^{i})_{i\in\omega}$ is eventually indiscernible.
\end{proof} 

We now present a collection of facts which will help us prove that the associated average measures along eventually indiscernible sequences always converge to a measure in $\mathfrak{M}_{x}(\mathcal{U})$ when the underlying theory is NIP. The first fact is elementary and left to the reader as an exercise. 

\begin{fact} Assume that $(\mu_{i})_{i \in \omega}$ is a sequence of Keisler measures in $\mathfrak{M}_{x}(\mathcal{U})$. If for every formula $\varphi(x) \in \mathcal{L}_{x}(\mathcal{U})$, $\lim_{i \to \infty} \mu_{i}(\varphi(x))$ converges, then $(\mu_{i})_{i \in \omega}$ converges to a measure in $\mathfrak{M}_{x}(\mathcal{U})$. 
\end{fact}

The next collection of facts can be found in \cite{HPS}. In particular, $(i)$ follows immediately from Lemma 2.10 while $(ii)$ and $(iii)$ are from Corollary 2.14. The proof of Lemma 2.10 is non-trivial and is an interpretation of results in \cite{Ben}. Implicitly, our proof uses the fact that the randomization of an NIP theory is NIP.

\begin{fact}[T is NIP]\label{HPSFact} Suppose that $\lambda \in \mathfrak{M}_{(x_i)_{i \in \omega}}$ where $|x_i| =|x_j|$ for each $i,j < \omega$.  $\lambda$ is said to be \textbf{$M$-indiscernible} if for every increasing sequence of indices $i_0,...,i_n$ and any formula $\varphi(x_{i_0},...,x_{i_{n}})$ in 
$\mathcal{L}_{(x_i)_{i \in \omega}}(M)$, we have that
\begin{equation*}
\lambda(\varphi(x_{i_0},...,x_{i_n})) = \lambda(\varphi(x_{0},...,x_{n})).
\end{equation*} 
Let $\mu, \nu \in \mathfrak{M}_{x}(\mathcal{U})$ such that $\mu,\nu$ are invariant over $M$. The following statements are true.
\begin{enumerate}[($i$)]
    \item If $\lambda$ is $M$-indiscernible, then for any formula $\varphi(x,b) \in \mathcal{L}_{x}(\mathcal{U})$, we have that $\lim_{i \to \infty} \lambda(\varphi(x_i,b))$ exists. 
    \item The measures $\mu^{(\omega)}$ and $\nu^{(\omega)}$ are $M$-indiscernible. 
    \item If $\mu^{(\omega)}|_{M} = \nu^{(\omega)}|_{M}$, then $\mu = \nu$. 
\end{enumerate}
\end{fact}

We now establish a formal connection between eventually indiscernible sequences of tuples and indiscernible measures. We use this connection to show that the eventually indiscernible sequences converges to a measure in $\mathfrak{M}_{x}(\mathcal{U})$. 

\begin{proposition}\label{converge} Let $(\overline{c}_i)_{i \in \omega}$ be a sequence of points in $(\mathcal{U}^{x})^{<\omega}$. If $(\overline{c}_i)_{i \in \omega}$ is an eventually indiscernible sequence over some model $M$, then the sequence $(\Av(\overline{c}_i))_{i \in \omega}$ converges in $\mathfrak{M}_{x}(\mathcal{U})$. 
\end{proposition}

\begin{proof} Assume not. Then there exists some formula $\psi(x,b)$ in $\mathcal{L}_{x}(\mathcal{U})$, some $\epsilon_{0} >0$, and some subsequence $(\overline{c}_i')_{i \in \omega}$ of $(\overline{c}_{i})_{i \in \omega}$ such that for each natural number $i$, 
\begin{equation*}
|\Av(\overline{c}_i')(\psi(x;b)) - \Av(\overline{c}_{i+1}')(\psi(x;b))| > \epsilon_0.
\end{equation*}
It is clear that $(\overline{c}_{i}')_{i \in \omega}$ is also eventually indiscernible over $M$. We now aim to contradict $(i)$ of Fact \ref{HPSFact} via (topological) compactness of the space $\mathfrak{M}_{\omega}(\mathcal{U}) : = \mathfrak{M}_{(x_{i})_{i \in \omega}}(\mathcal{U})$. For any formula $\varphi(x_{i_{0}},...,x_{i_{k}}) \in \mathcal{L}_{(x_i)_{i \in \omega}}(M)$, we let $r_{\varphi}$ be the unique real number such that for every $\epsilon > 0$, there exists an $N_{\epsilon,\varphi}$ so that for any $n_k > ... >n_0 > N_{\epsilon,\varphi}$ we have 
\begin{equation*} | \varphi_{c}(\overline{c}'_{n_0},...,\overline{c}'_{n_k}) - r_{\varphi} | < \epsilon. 
\end{equation*}
Since the sequence $(\overline{c}_{i}')_{i\in \omega}$ is eventually indiscernible over $M$, $r_{\varphi}$ exists for each $\varphi(\overline{x}) \in \mathcal{L}_{(x_i)_{i \in \omega}}(M)$. Now, for every $\varphi(\overline{x}) \in \mathcal{L}_{(x_i)_{i \in \omega}}(M)$ and $\epsilon >0$, we define the following family of closed subsets of $\mathfrak{M}_{\omega}(\mathcal{U})$; 
\begin{equation*} C_{\epsilon,\varphi} = \Big\{ \lambda \in \mathfrak{M}_{\omega}(\mathcal{U}): r_{\varphi} - \epsilon \leq \lambda(\varphi(\overline{x})) \leq r_{\varphi} + \epsilon \Big\}.
\end{equation*}
We also define another family of sets and argue that they are closed; let
\begin{equation*}
D_{i} = \Big\{\lambda \in \mathfrak{M}_{\omega}(\mathcal{U}) : |\lambda(\psi(x_i,b)) - \lambda(\psi(x_{i+1},b))| \geq \frac{\epsilon_{0}}{2}\Big\}.
\end{equation*}
Notice that $D_{i}$ is closed since for every natural number $i$, the evaluation map $E_{i}: \mathfrak{M}_{\omega}(\mathcal{U}) \to [0,1]$ via $E_{i}(\lambda) = \lambda(\varphi(x_i,b))$ is continuous. Indeed, define $F_{i} = E_{i} - E_{i+1}$ and $H_{i} = E_{i+1} - E_{i}$.  Then we have $D_{i} = F_{i}^{-1}([\frac{\epsilon_{0}}{2},1]) \cup H_{i}^{-1}([\frac{\epsilon_{0}}{2},1])$ and so $D_{i}$ is a union of two closed sets and therefore closed. Using $(\overline{c}_{i}')_{i \in \omega}$, the collection $\Phi = \{C_{\epsilon,\varphi}: \epsilon > 0, \varphi(\overline{x}) \in \mathcal{L}_{\omega}(M)\}\cup\{D_{i}: i \in \omega\}$ has the finite intersection property. Therefore, there exists some $\lambda \in \mathfrak{M}_{\omega}(\mathcal{U})$ in the intersection of all the sets in $\Phi$. Moreover, $\lambda$ is $M$-indiscernible by construction. Since $\lambda$ is in $D_{i}$ for each $i$, its existence contradicts $(i)$ of Fact \ref{HPSFact}. 
\end{proof} 

\subsection{Smooth sequences} In this subsection, we define the notion of a smooth sequence and prove the main theorem. If $\mu$ is a global $M$-invariant measure, then a smooth sequence is a collection of models and measures meant to replicate a Morley sequence. The ideology is the following: A Morley sequence in $p$ over $M$ is to the infinite type $p^{\omega}|_{M}$ as a smooth sequence in $\mu$ over $M$ is to the measure $\mu^{(\omega)}|_{M}$. We now provide the formal definition.

\begin{definition}Let $\mu \in \mathfrak{M}_{x}(\mathcal{U})$ and assume that $\mu$ is invariant over some small model $M$. Then, a \textbf{smooth sequence in $\mu$ over $M$} is a sequence of pairs of measures and small models, $(\mu_i,N_i)_{i \in \omega}$, such that: 
\begin{enumerate}[$(i)$]
    \item $M \prec N_0$ and $N_i \prec N_{i +1}$ and each $N_i$ is small.
    \item $\mu_{i}$ is smooth over $N_i$. 
    \item $\mu_{0}|_M = \mu|_M$ and for $i > 0$, $\mu_{i}|_{N_{i-1}} = \mu|_{N_{i-1}}$. 
\end{enumerate}
Furthermore, we define $\bigotimes_{i=0}^{\omega} \mu_{i} = \bigcup_{i =0}^{\omega} \bigotimes_{i=0}^{n}\mu_i$ which is an element of $\mathfrak{M}_{(x_i)_{i \in \omega}}(\mathcal{U})$. We let $N_{\omega} = \bigcup_{i \in \omega} N_{i}$. Notice that for each $i \in \omega$, the measure $\mu_{i}$ is smooth over $N_{\omega}$.
\end{definition}

\begin{proposition}\label{existence} If $T$ is a countable NIP theory, $\mu \in \mathfrak{M}_{x}(\mathcal{U})$, and $\mu$ is invariant over $M$ where $|M|=\aleph_0$, then there exists a smooth sequence $(\mu_{i},N_i)_{i\in\omega}$ in $\mu$ over $M$ such that each $N_{i}$ is countable. 
\end{proposition}
\begin{proof} This follows directly from Proposition \ref{m:countable}.
\end{proof}

\begin{proposition}[T is NIP]\label{smoothinv} Assume that $\mu \in \mathfrak{M}_{x}(\mathcal{U})$ and $\mu$ is $M$-invariant. Let $(\mu_i,N_i)_{i \in \omega}$ be a smooth sequence in $\mu$ over $M$. Then, $\bigotimes_{i=0}^{\omega} \mu_{i} |_{M} = \mu^{(\omega)}|_{M}$. Hence, $\bigotimes_{i=0}^{\omega} \mu_i$ is $M$-indiscernible.
\end{proposition}

\begin{proof}
We prove this via induction on formulas in $\mathcal{L}_{(x_i)_{i \in \omega}}(\mathcal{U})$. For our base case, it is true by construction that $\mu_{0}|_{M} = \mu|_{M}$. For our induction hypothesis, we assume that $\mu^{(k-1)}|_{M} = \bigotimes_{i=0}^{k-1} \mu_{i}|_{M}$. For ease of notation, we set $\lambda = \bigotimes_{i=0}^{k-1} \mu_{i}$ and show the induction step: Let $\varphi(x_0,...,x_k)$ be any formula in $\mathcal{L}_{x_0,...,x_k}(M)$. Since the product of smooth measures is smooth (by $(iii)$ of Fact \ref{KM:imp2}), we have that $\lambda$ is smooth over $N_{k-1}$. In particular, $\lambda$ is invariant over $N_{k-1}$. We let $\overline{x} = (x_0,...,x_{k-1})$ and $\theta(x_{k};\overline{x}) = \varphi(x_{0},...,x_{k})$. We consider the following computation followed by a list of justifications. 

\begin{equation*}
    \mu_{k} \otimes \lambda (\varphi(x_0,...,x_{k})) = \int_{S_{\overline{x}}(N_{k})} F_{\mu_{k}}^{\theta} d(\lambda|_{N_{k}}) \overset{(a)}{=} \int_{S_{x_{k}}(N_{k})} F_{\lambda}^{\theta^*}d(\mu_{k}|_{N_{k}})
\end{equation*}
\begin{equation*}
 \overset{(b)}{=} \int_{S_{x_{k}}(N_{k-1})}F_{\lambda}^{\theta^{*}}d(\mu_{k}|_{N_{k-1}}) \overset{(c)}{=} \int_{S_{x_{k}}(N_{k-1})}F_{\lambda}^{\theta^{*}}d(\mu|_{N_{k-1}}) \overset{(a)}{=}\int_{S_{\overline{x}}(N_{k-1})}F_{\mu}^{\theta}d(\lambda|_{N_{k-1}})
\end{equation*} 
\begin{equation*}
\overset{(d)}{=} \int_{S_{\overline{x}}(M)} F_{\mu}^{\theta} d(\lambda|_M) \overset{(e)}{=} \int_{S_{\overline{x}}(M)} F_{\mu}^{\theta} d(\mu^{(k-1)}|_M) =  \mu \otimes \mu^{(k-1)} (\varphi(x_0,...,x_{k})).
\end{equation*} 
We provide the following justifications: 
\begin{enumerate}[($a$)]
\item Smooth measures commute with invariant measures. 
\item Changing space of integration since $\lambda$ is invariant over $N_{k-1}$. 
\item By construction of smooth sequences, we have that $\mu_{k}|_{N_{k - 1}} = \mu|_{N_{k -1}}$. 
\item Changing space of integration since $\mu$ is invariant over $M$. 
\item By our induction hypothesis. \qedhere
\end{enumerate}
\end{proof}

We now begin the proof of our main theorem. Again, the proof is similar to both the generically stable case in the previous section and even more so to the proof of Lemma 2.8 in \cite{Invariant}. Here, the major difference is that we replace the Morley sequence in that proof with a countable model, $N_{\omega}$, which ``contains" a smooth sequence in $\mu$ over $M$. Then we find a sequence of elements in $(M^{x})^{< \omega}$ such that the associated average measures converge to $\mu|_{N_{\omega}}$ in $\mathfrak{M}_{x}(N_{\omega})$. After choosing an eventually indiscernible subsequence, we know from our NIP assumption that this new sequence converges to a global measure $\nu$ in $\mathfrak{M}_{x}(\mathcal{U})$. Finally, we demonstrate that $\nu^{(\omega)}|_{M} = \mu^{(\omega)}|_{M}$ which completes the proof. 

\begin{theorem}[$T$ is NIP] Let $\mu$ be finitely satisfiable over a countable model $M$. Then there exists a sequence $(\overline{a})_{i \in \omega}$ of elements, each in $(M^{x})^{<\omega}$, such that for any $\theta(x) \in \mathcal{L}_{x}(\mathcal{U})$, we have that, 
\begin{equation*}
    \lim_{i \to \infty} \Av(\overline{a}_{i})(\theta(x)) = \mu(\theta(x)). 
\end{equation*} 
\end{theorem}
\begin{proof} Choose a smooth sequence $(\mu_{i},N_i)_{i \in \omega}$ in $\mu$ over $M$. By Proposition \ref{existence} we may choose this sequence so that for each $i \in \omega$, $N_i$ is countable. In particular, this implies that $N_{\omega}$ is a countable model. We begin by constructing a sequence of elements $(\overline{a}_{i})_{i \in \omega}$ in $(M^{x})^{< \omega}$ such that $(\Av(\overline{a}_{i})|_{N_{\omega}})_{i \in \omega}$ converges to $\mu|_{N_{\omega}}$ in $\mathfrak{M}_{x}(N_{\omega})$. Since $N_{\omega}$ is countable, we let $(\theta_{i}(x))_{i \in \omega}$ be an enumeration of the formulas in $\mathcal{L}_{x}(N_{\omega})$. Since $\mu$ is finitely satisfiable over $M$, we can find we find $\overline{a}_{k} \in (M^{x})^{<\omega}$ such that for any $j \leq k$, we have that, 
\begin{equation*}
    |\mu(\theta_{j}(x)) - \Av(\overline{a}_{k})(\theta_{j}(x))| < \frac{1}{k}. 
\end{equation*}
By construction, it is clear that the sequence $(\Av(\overline{a}_i)|_{N_{\omega}})_{i\in \omega}$ converges to $\mu|_{N_{\omega}}$ in $\mathfrak{M}_{x}(N_{\omega}$). Now, we let $(\overline{c}_i)_{i \in \omega}$ be a subsequence of $(\overline{a}_i)_{i \in \omega}$ so that $(\overline{c}_i)_{i \in \omega}$ is eventually indiscernible over $N_{\omega}$. Then the sequence $(\Av(\overline{c}_i))_{i \in \omega}$ converges in $\mathfrak{M}_{x}(\mathcal{U})$ by Proposition \ref{converge}. Assume that $(\Av(\overline{c}_{i}))_{i \in \omega}$ converges to some measure $\nu \in \mathfrak{M}_{x}(\mathcal{U})$. Hence, $\nu$ is finitely satisfiable over $M$ by $(i)$ of Proposition \ref{finitesat} and therefore $\nu$ is invariant over $M$. We show that $\nu^{(\omega)}|_{M} = \mu^{(\omega)}|_{M}$. This will conclude the proof by $(iii)$ of Fact \ref{HPSFact}. 

Since $(\overline{c}_{i})_{i \in \omega}$ is a subsequence of $(\overline{a}_{i})_{i \in \omega}$, it follows that $\nu|_{N_{\omega}} = \mu|_{N_{\omega}}$ and therefore $\nu|_{M} = \mu|_{M}$. We now proceed by induction. Assume that $\nu^{(k-1)}|_{M} = \mu^{(k-1)}|_{M}$.  Fix $\varphi(x_0,...,x_{k})$ in $\mathcal{L}_{x_0,...,x_k}(M)$.  For ease of notation, set $\lambda = \bigotimes_{i=0}^{k-1} \mu_{i}$. We recall that $\lambda$ is smooth over $N_{\omega}$ (see Fact \ref{KM:imp2}).  By Proposition \ref{smoothinv}, $\mu^{(k-1)}|_{M} = \lambda|_{M}$. We let $\overline{x} = (x_0,...,x_{k-1})$ and let $\theta(x_{k};\overline{x}) = \varphi(x_0,...,x_k)$. We now consider the critical computation followed a small glossary of justifications.
\begin{equation*}
    \nu^{(k)}(\varphi(x_0,...,x_{k})) = \int_{S_{\overline{x}}(M)} F_{\nu}^{\theta} d(\nu^{(k-1)}|_{M}) \overset{(a)}{=} \int_{S_{\overline{x}}(M)} F_{\nu}^{\theta} d(\mu^{(k-1)}|_M)
\end{equation*}
\begin{equation*}
    \overset{(b)}{=} \int_{S_{\overline{x}}(M)} F_{\nu}^{\theta}d(\lambda|_{M}) \overset{(c)}{=} \int_{S_{\overline{x}}(N_{\omega})} F_{\nu}^{\theta}d(\lambda|_{N_{\omega}}) \overset{(d)}{=} \int_{S_{x_{k}}(N_{\omega})} F_{\lambda}^{\theta^*}d(\nu|_{N_{\omega}})
\end{equation*}
\begin{equation*}
    \overset{(e)}{=} \int_{S_{x_{k}}(N_{\omega})} F_{\lambda}^{\theta^*}d(\mu|_{N_\omega}) \overset{(d)}{=} \int_{S_{\overline{x}}(N_{\omega})} F_{\mu}^{\theta} d(\lambda|_{N_{\omega}}) \overset{(c)}{=} \int_{S_{\overline{x}}(M)} F_{\mu}^{\theta} d(\lambda|_{M}) 
\end{equation*}
\begin{equation*}
 \overset{(b)}{=}\int_{S_{\overline{x}}(M)} F_{\mu}^{\theta} d(\mu^{(k-1)}|_{M}) = \mu^{(k)}(\varphi(x_0,...,x_{k})).
\end{equation*}
We provide the following justifications: 
\begin{enumerate}[(a)]
\item Induction hypothesis.
\item $\mu^{(k-1)}|_M = \lambda|_M$. 
\item Changing the space of integration.
\item Smooth measures commute with invariant measures. 
\item $\nu|_{N_{\omega}} = \mu|_{N_{\omega}}$ \qedhere
\end{enumerate}
\end{proof} 

We now observe that we have another proof of the theorem that global measures in NIP theories which are definable and finitely satisfiable are also finitely approximated.   

\begin{corollary} If $T'$ is a countable or uncountable NIP theory and $\mu$ is $\dfs$ over $M$, then $\mu$ is finitely approximated over $M$. 
\end{corollary}
\begin{proof} After restricting to a countable language, we still have a $\dfs$ measures (by \cite[Proposition 2.9]{CoGan}).  By Proposition \ref{m:countable}, $\mu$ restricted to this language is $\dfs$ over a countable model, $M_0$. By the previous result, $\mu$ is sequentially approximated over $M_0$. Since $\mu$ is also definable, an application of Proposition \ref{Mazur} yields the result.   
\end{proof}

\begin{observation} Assume that $\mu \in \mathfrak{M}_{x}(\mathcal{U})$ and let $M$ be a small elementary submodel. Then, $\mu$ is sequentially approximated over $M$ if

\begin{enumerate}
\item  $T$ is stable, and $\mu$ is invariant over $M$,
\item  $T$ is NIP, $|M| = \aleph_0$, and $\mu$ is finitely satisfiable over $M$, or
\item  $\mu$ is finitely approximated over $M$.
\end{enumerate}
\end{observation}

Finally, one may ask what happens in the local context. We remark that there exists two proofs for a local version of Theorem T2 which both rely on an important result of Bourgain, Fremlin, and Talagrand whose connection to model theory is (by now) well-known (e.g. \cite{IBFT,SimonBFT, Khanaki1,GannNIP}). Chronologically, the first proof of the following theorem is implicit in the work of Khanaki (see \cite[Remark 3.21, Theorem 3.26]{Khanaki1}) (through the observation that measures are types over models of the randomization in continuous model theory and \cite[Proposition 1.1]{Ben2}), 

\begin{theorem}\label{Khanaki}Suppose $\mu$ is a Keisler measure in $\mathfrak{M}_{\varphi}(\mathcal{U})$, $\mu$ is finitely satisfiable over $M$ where $|M| = \aleph_0$, and $\varphi(x,y)$ is an NIP formula. Then there exists a sequence of points $(\overline{a}_{i})_{i \in \omega}$ in $(M^{x})^{< \omega}$ such that for each $b \in \mathcal{U}^{y}$, 
\begin{equation*}\lim_{i \to \infty} \Av(\overline{a}_i)(\varphi(x,b)) = \mu(\varphi(x,b)).
\end{equation*}
\end{theorem}
\noindent There is another proof for the case of just Keisler measures via the VC theorem (see \cite[Lemma 4.7]{GannNIP}) which came later. 

\subsection{Smooth sequences and generically stable measures in NIP theories}

We now give an equivalent characterization for generically stable measures in NIP theories. We invite the reader to review the definition of a generically stable type prior to reading this section. Recall the following theorem due to Hrushovski, Pillay, and Simon \cite[Theorem 3.2]{HPS}. 

\begin{theorem}[T is NIP]\label{genstab:equiv} Assume that $\mu \in \mathfrak{M}_{x}(\mathcal{U})$. Then the following are equivalent. 
\begin{enumerate}[($i$)]
\item $\mu$ is dfs. 
\item $\mu$ is finitely approximated. 
\item $\mu$ is fim (see \cite[Definition 2.7]{HPS}).
\item $\mu$ is invariant and $\mu_{x} \otimes \mu_{y} = \mu_{y} \otimes \mu_{x}$. 
\end{enumerate}
Moreover, a Keisler measure (in an NIP theory) is called \textbf{generically stable} if it satisfies any/all of $(i) - (iv)$. 
\end{theorem}

We will now show that smooth sequences can also give a characterization of generically stable measures in NIP theories.

\begin{lemma}[T is NIP] Let $\mu \in \mathfrak{M}_{x}(\mathcal{U})$. Suppose that $\mu$ is generically stable over $M$. For any smooth sequence $(\mu_i,N_i)_{i \in \omega}$ in $\mu$ over $M$, we have that $\lim_{i \to \infty} \mu_i = \mu$ in $\mathfrak{M}_x(\mathcal{U})$. 
\end{lemma}

\begin{proof} Since $(\mu_i,N_i)_{i \in \omega}$ is a smooth sequence in $\mu$ over $M$, the measure $\bigotimes_{i=0}^{\omega} \mu_i$ is indiscernible over $M$ by Proposition \ref{smoothinv}. By $(i)$ of Fact \ref{HPSFact}, we know that $\lim_{i \to \infty} \mu_{i} = \nu$ for some $\nu \in \mathfrak{M}_{x}(\mathcal{U})$. Since each $\mu_i$ is finitely satisfiable over $N_i$, it follows that $\nu$ is finitely satisfiable over $N_{\omega}$. By $(iii)$ of Fact \ref{HPSFact}, it is enough to show that $\nu^{(\omega)}|_{N_{\omega}} = \mu^{(\omega)}|_{N_{\omega}}$. The base case is trivial. Assume that $\nu^{(k-1)}|_{N_{\omega}} = \mu^{(k-1)}|_{N_{\omega}}$. Fix $\varphi(x_0,...,x_k) \in \mathcal{L}_{x_0,...,x_k}(N_{\omega})$ and $\epsilon > 0$. Let $\overline{x} = (x_0,...,x_{k-1})$ and $\theta(x_k;\overline{x}) = \varphi(x_0,...,x_k)$. Since $\mu$ is generically stable over $M$, $\mu^{(k-1)}$ is generically stable over $M$ ($(v)$ of Fact \ref{KM:imp2}) and so also definable over $N_{\omega}$. Therefore by $(v)$ of Fact \ref{KM:imp}, there exists formulas $\psi_1(x_{k}),...,\psi_{n}(x_{k}) \in \mathcal{L}_{x_{k}}({N_{\omega}})$ and real numbers $r_1,...,r_n \in [0,1]$ so that 

\begin{equation*}
\sup_{q \in S_{x_{k}}(N_{\omega})} | F_{\mu^{(k-1)}}^{\theta^{*}}(q) - \sum_{i=1}^{n} r_i \mathbf{1}_{\psi_i(x_k)}(q)| < \epsilon. 
\end{equation*} 
Consider the following sequence of equations followed by a short list of justifications.

\begin{equation*}
    \nu^{(k)}(\varphi(x_0,...,x_{k})) = \int_{S_{\bar{x}}(N_{\omega})} F_{\nu}^{\theta} d( \nu^{(k-1)}|_{N_{\omega}}) \overset{(a)}{=} 
    \int_{S_{\bar{x}}(N_{\omega})} F_{\nu}^{\theta} d(\mu^{(k-1)}|_{N_{\omega}})
\end{equation*}
\begin{equation*}
    \overset{(b)}{=} \int_{S_{x_{k}}(N_{\omega})} F_{\mu^{(k-1)}}^{\theta^{*}} d(\nu|_{N_{\omega}}) \approx_{\epsilon} \int_{S_{x_{k}}(N_{\omega})} \sum_{i=1}^{n} r_i \mathbf{1}_{\psi_{i}(x_{k})} d(\nu|_{N_{\omega}})
\end{equation*}
\begin{equation*}
    = \sum_{i=1}^{n} r_{i} \nu(\psi_{i}(x_{k})) \overset{(c)}{=} \sum_{i=1}^{n} r_{i} \mu(\psi_{i}(x_{k})) = \int_{S_{x_{k}}(N_{\omega})} \sum_{i=1}^{n} r_i \mathbf{1}_{\psi_{i}(x_{k})} d(\mu|_{N_{\omega}})
\end{equation*}
\begin{equation*}
    \approx_{\epsilon} \int_{S_{x_{k}}(N_{\omega})} F_{\mu^{(k-1)}}^{\theta^{*}} d(\mu|_{N_{\omega}}) \overset{(b)}{=} \int_{S_{x_{k}}(N_{\omega})} F_{\mu}^{\theta} d(\mu^{(k-1)}|_{N_{\omega}}) =  \mu^{(k)}(\varphi(x_0,...,x_{k})). 
\end{equation*}
\begin{enumerate}[(a)]
\item Induction hypothesis. 
\item (T is NIP) Generically stable measures commute with invariant measures (see $(b)$ of Fact \ref{KM:imp2}). 
\item Base case.
\end{enumerate}
As $\epsilon$ was arbitrary, this proves the result.
\end{proof}

\begin{lemma}[T is NIP] Assume that $\mu$ is $M$-invariant. If for every smooth sequence $(\mu_{i},N_i)_{i \in \mathbb{N}}$ in $\mu$ over $M$, we have that $\lim_{i \to \infty} \mu_{i} = \mu$, then $\mu$ is generically stable over $M$.
\end{lemma}
\begin{proof}

Since $T$ is NIP, all invariant measures are Borel definable. By Theorem \ref{genstab:equiv}, it suffices to show that $\mu$ commutes with itself, i.e. $\mu_x \otimes \mu_y = \mu_y \otimes \mu_x$. Fix $\varphi(x,y) \in \mathcal{L}_{x,y}(\mathcal{U})$. Let $M_1$ be a small model such that $ M \prec M_1$ and $M_1$ contains all the parameters from $\varphi(x,y)$. We choose a smooth sequence $(\mu_{i,x}; N_i)_{i \in \omega}$ in $\mu_{x}$ over $M_1$ and let $N_\omega = \bigcup_{i \in \omega} N_i$. By construction, the sequence $(\mu_{i,x},N_i)_{i \in \omega}$ is a smooth sequence in $\mu_{x}$ over $M$. Consider the following computation.

\begin{equation*}
\mu_{x}\otimes\mu_{y}(\varphi(x,y))=
\int_{S_{y}(M_{1})}F_{\mu_{x}}^{\varphi}d(\mu_{y}|_{M_{1}}) \overset{(a)}{=}
\int_{S_{y}(N_{\omega})}F_{\mu_{x}}^{\varphi}d(\mu_{y}|_{N_{\omega}}) 
\end{equation*}
\begin{equation*}
\overset{(b)}{=}\lim_{i\to\infty}\int_{S_{y}(N_{\omega})}F_{\mu_{i,x}}^{\varphi}d(\mu_{y}|_{N_{\omega}})
\overset{(c)}{=} \lim_{i\to\infty}\int_{S_{x}(N_{\omega})}F_{\mu_{y}}^{\varphi^{*}}d(\mu_{i,x}|_{N_{\omega}})
\end{equation*}
\begin{equation*} 
\overset{(d)}{=} \lim_{i \to \infty} \int_{S_{x}(M_{1})} F_{\mu_y}^{\varphi^{*}} d(\mu_{i,x}|_{M_1})  \overset{(e)}{=} \lim_{i \to \infty} \int_{S_{x}(M_1)} F_{\mu_{y}}^{\varphi^{*}} d(\mu_{x}|_{M_1})
\end{equation*} 
\begin{equation*} = \int_{S_{x}(M_1)} F_{\mu_{y}}^{\varphi^{*}} d(\mu_{x}|_{M_1}) = \mu_{y} \otimes \mu_{x} (\varphi(x,y)). 
\end{equation*}

\noindent We provide a list of the following justifications:
\begin{enumerate}[$(a)$]
\item Changing the space of integration. 
\item Dominated convergence theorem. 
\item Smooth measures commute with Borel definable measures.
\item Since $\mu_{y}$ is $M_1$ invariant. 
\item Since $\mu_{i,x}|_{M_{1}} = \mu_{x}|_{M_1}$ for any $i \in \omega$. \qedhere
\end{enumerate} 
\end{proof}

\begin{theorem}[T is NIP] Let $\mu \in \mathfrak{M}_{x}(\mathcal{U})$. Then the following are equivalent: 
\begin{enumerate} 
\item $\mu$ is generically stable over $M$. 
\item For any smooth sequence $(\mu_i,N_i)_{i \in \omega}$ in $\mu$ over $M$,
\begin{equation*} \lim_{i \to \infty} \mu_i = \mu \text{ in $\mathfrak{M}_{x}(\mathcal{U})$.}
\end{equation*} 
\end{enumerate} 
\end{theorem}

\begin{proof} Follows directly from the previous two lemmas.
\end{proof} 

\section{Local Measures revisited}
We generalize the main theorem of \cite{GannNIP}. Fix a partitioned NIP formula $\varphi(x,y)$ and let $\mu$ be a $\varphi$-measure. In \cite{GannNIP}, we proved two main theorems. We showed that if $\varphi(x,y)$ is an NIP formula and $\mu$ is $\varphi$-definable and finitely satisfiable over a \textbf{countable} model $M$, then $\mu$ is $\varphi$-finitely approximated. We then proved that if $\mu$ is definable and finitely satisfiable over any small model $M$, then $\mu$ is finitely approximated in $M$ by reducing to the previous theorem. But this was somewhat unsatisfactory and the following question was left open: if $\mu$ is $\varphi$-definable and finitely satisfiable over a \textbf{small} model, then is $\mu$ $\varphi$-finitely approximated? We give a positive answer to this question by modifying one of the important technical lemmas in the proof. Let us first recall some definitions. 

\begin{definition}\label{local} Fix $\mathcal{U}$ and a formula $\varphi(x,y)$ in $\mathcal{L}(\mathcal{U})$. 
\begin{enumerate} 
\item $\mathcal{L}_{\varphi}(\mathcal{U})$ denotes the Boolean algebra of definable sets of $\mathcal{U}^{x}$ generated by the collection $\{\varphi(x,b): b \in \mathcal{U}\}$. 
\item A $\varphi$-measure is a finitely additive measure on the Boolean algebra $\mathcal{L}_{\varphi}(\mathcal{U})$. 
\item The collection of all $\varphi$-measures is denoted $\mathfrak{M}_{\varphi}(\mathcal{U})$. 
\item Let $M \prec \mathcal{U}$ and assume that $M$ contains all the parameters from $\varphi(x,y)$. For any $\mu \in \mathfrak{M}_{\varphi}(\mathcal{U})$, we say that $\mu$ is $(M,\varphi)$-invariant if for any $b,c \in \mathcal{U}^{y}$ such that $\tp(b/M) = \tp(c/M)$, we have that $\mu(\varphi(x,b)) = \mu(\varphi(x,c))$. 
\item Let $\mu \in \mathfrak{M}_{\varphi}(M)$. If $\mu$ is $(M,\varphi)$-invariant, then we define can the fiber map $F_{\mu}^{\varphi}: S_{y}(M) \to [0,1]$ via $F_{\mu,M}^{\varphi}(q) = \mu(\varphi(x,b))$ where $b \models q|_M$. When $M$ is clear from context, we write $F_{\mu,M}^{\varphi}$ simply as $F_{\mu}^{\varphi}$. 
\item Let $\mu \in \mathfrak{M}_{\varphi}(\mathcal{U})$. Then $\mu$ is said to be $\varphi$-definable if the map $F_{\mu,M}^{\varphi}: S_{y}(M) \to [0,1]$ is continuous.
\item Let $\mu \in \mathfrak{M}_{\varphi}(\mathcal{U})$. Then $\mu$ is said to be definable if for any formula $\theta(x,\overline{y})$ in the algebra generated by $\{\varphi(x,y_i): i \in \mathbb{N}\}$, $\mu$ is $(M,\theta)$-invariant and the map $F_{\mu}^{\varphi}:S_{y}(M) \to [0,1]$ is continuous. 
\item For any $\mu \in \mathfrak{M}_{\varphi}(\mathcal{U})$, $\mu$ is said to be finitely satisfiable in $M$ if for every $\theta(x) \in \mathcal{L}_{\varphi}(\mathcal{U})$ such that $\mu(\theta(x)) > 0$, there exists some $a \in M$ so that $\mathcal{U} \models \theta(a)$. 
\item For each $a \in M$ we let $F_{a}^{\varphi}: S_{y}(M) \to [0,1]$ via $F_{a}^{\varphi} = \mathbf{1}_{\varphi(a,y)}$. We denote the collection of such functions as $\mathbb{F}_{M}$. We let $\conv(\mathbb{F}_M)$ be the collection of convex combinations of elements in $\mathbb{F}_{M}$. We let $F = [0,1]^{S_{y}(M)}$ endowed with the Tychonoff topology and if $A \subset F$, we let $\cl(A)$ denote its closure in this space and so the set $\cl(\conv(A))$ is well-defined. 
\end{enumerate}
\end{definition} 

\noindent Recall the following facts about $\varphi$-measures which can be found in \cite{GannNIP}.

\begin{fact}\label{local:facts} Let $\mu \in \mathfrak{M}_{\varphi}(\mathcal{U})$ and $M \prec \mathcal{U}$.
\begin{enumerate}[$(i)$]
\item If $\mu$ is finitely satisfiable or $\varphi$-definable over $M$ then $\mu$ is $(M,\varphi)$-invariant.
\item If $\mu$ is $\varphi$-definable over $M$ then $\mu$ is $(M_{0},\varphi)$-invariant for some $M_0 \prec M$ such that $|M_0| = \aleph_0$. 
\item If $\mu$ is finitely satisfiable over $M$ then $F_{\mu,M}^{\varphi}$ is in $\cl(\conv(\mathbb{F}_{M}))$.
\item If $|M| = \aleph_0$ and $\varphi(x,y)$ is NIP, there exists a sequence of elements $(g_i)_{i \in \omega}$ with each $g_i \in \conv(\mathbb{F}_M)$ so that $\lim_{i \to \infty} g_i = F_{\mu,M}^{\varphi}$. 
\end{enumerate}
\end{fact}

\noindent The following lemma is essentially the \textit{missing lemma} from \cite{GannNIP}. The missed observation is that one can consider finitely many parameters at once (instead of a single parameter). 

\begin{lemma}\label{meas:lemma} Suppose that $\mu \in \mathfrak{M}_{\varphi}(\mathcal{U})$ and $\mu$ is finitely satisfiable in a small submodel $N$ and $(M,\varphi)$-invariant. Then the map $F_{\mu,M}^{\varphi} \in \cl(\conv(\mathbb{F}_M))$. 
\end{lemma}
\begin{proof} The proof is similar to the proof for types \cite[Lemma 2.18]{Sibook} as well as the proof for measures \cite[Proposition 4.13]{GannNIP} (which has both a stronger assumption and conclusion). It suffices to show that for any finite collection of types $p_1,...,p_n \in S_{y}(M)$ and $\epsilon > 0$ there exists $\overline{a} \in (M^{x})^{< \omega}$ such that $F_{\Av(\overline{a}),M}^{\varphi}(p_i) \approx_{\epsilon} F_{\mu,M}^{\varphi}(p_i)$ for each $i \leq n$. 

Fix $p_1,...,p_n \in S_{y}(M)$ and $\epsilon >0$. Choose $b_i \models p_i$ for $i \leq n$. Let $q = \tp(N/M) \in S_{|N|}(M)$. Let $\hat{q} \in S_{|N|}(\mathcal{U})$ such that $\hat{q} \supset q$ and $\hat{q}$ is finitely satisfiable in $M$, i.e. $\hat{q}$ is a global coheir of $q$. Let $N_{1} \models \hat{q}|_{Mb_1,...,b_n}$. 

By compactness, there exists elements $b_1',...,b_n' \in \mathcal{U}$ such that $\tp(N_1 b_1,...,b_n/M) = tp(Nb_1',...,b_n'/M)$. Since $\mu$ is $(M,\varphi)$-invariant, we have that 
\begin{equation*} F_{\mu,M}^{\varphi}(p_i) = \mu(\varphi(x,b_i)) = \mu(\varphi(x,b'_i)),
\end{equation*} 
for each $i \leq n$. Since $\mu$ is finitely satisfiable in $N$, there exists some $m$ and $\overline{c} \in (N^{x})^{m}$ such that $\Av(\overline{c})(\varphi(x,b'_i)) \approx_{\epsilon} \mu(\varphi(x,b_i'))$ for $i \leq n$. Let $B_i =\{j \leq m: \models \varphi(c_j,b'_i)\}$. Now consider the formula
\begin{equation*}
    \theta(x_1,...,x_m,y_1,...y_n) = \bigwedge_{i \leq n} \Big( \bigwedge_{j\in B_i} \varphi(x_{j},y_{i}) \wedge \bigwedge_{j \not \in B_i} \neg \varphi(x_{j},y_{i}) \Big).
\end{equation*}
By construction $\theta(\overline{x},\overline{y}) \in \tp(\overline{c},\overline{b'}/M)$ and so for an appropriate choice of indices, $\theta(\overline{x},\overline{y}) \in \tp(Nb_1',...,b_n'/M)$. Hence $\theta(\overline{x},\overline{y}) \in \tp(N_1b_1,...,b_n/M)$ and so $\theta(\overline{x},\overline{b}) \in \tp(N_1/Mb_1,...,b_n) \subset \hat{q}$. Since $\hat{q}$ is finitely satisfiable in $M$, there exists $\overline{a} \in (M^{x})^{m}$ such that $\models \theta(\bar{a},\bar{b})$. By construction, we have that for any $i \leq n$, 
\begin{equation*}
    F_{\Av(\overline{a}),M}^{\varphi}(p_i) = \Av(\overline{a})(\varphi(x,b_i)) = \Av(\overline{c})(\varphi(x,b'_i)) \approx_{\epsilon} \mu(\varphi(x,b'_i)) = F_{\mu,M}^{\varphi}(p_i).
\end{equation*}
This concludes the proof.
\end{proof}

\begin{theorem}\label{main:Gan} Fix a formula $\varphi(x,y)$ and a small model $M$ containing all the parameters from $\varphi(x,y)$. Assume that $\mu \in \mathfrak{M}_{\varphi}(\mathcal{U})$. If
\begin{enumerate}
    \item $\varphi(x;y)$ is NIP,
    \item $\mu$ is $\varphi$-definable over $M$, 
    \item and $\mu$ is finitely satisfiable in $M$,
\end{enumerate}
Then for every $\epsilon > 0$, there exists $a_1,...,a_n \in M^{x}$ such that, 
\begin{equation*}
    \sup_{b \in \mathcal{U}^{y}}|\mu(\varphi(x,b)) - \Av(\overline{a})(\varphi(x,b))| < \epsilon. 
\end{equation*}
\end{theorem}

\begin{proof} We remark that the proof is similar to that of Proposition \ref{Mazur}. Since $\mu$ is $\varphi$-definable over $M$, $\mu$ is $(M_0,\varphi)$-invariant where $M_0$ is a countable submodel of $M$. By Lemma \ref{meas:lemma}, the map $F_{\mu,M_{0}}^{\varphi} \in \cl(\conv(\mathbb{F}_{M_0}))$. By Fact \ref{local:facts}, there exists a sequence $(g_i)_{i \in I}$ so that $\lim_{i \to \infty} g_i = F_{\mu,M_0}^{\varphi}$. By Mazur's lemma, for every $\epsilon > 0$, there exists a finite set $I \subset \mathbb{N}$ and positive real numbers $\{r_i: i \in I\}$ such that $\sum_{i \in I} r_i = 1$ and
\begin{equation*}
\sup_{q \in S_{y}(M_{0})} |F_{\mu,M_0}^{\varphi}(q) - \sum_{i \in I} r_i g_i(q)| < \epsilon. 
\end{equation*}
The map $\sum_{i \in I} r_i g_{i}$ can clearly be uniformly approximated by an average function. More explicitly, there exists $ \overline{d} \in (M^{x})^{<\omega}$ such that 
\begin{equation*} \sup_{q \in S_{y}(M)} |\sum_{i \in I} r_i g_i (q) - F^{\varphi}_{\Av(\overline{d}),M}(q)| <\epsilon. 
\end{equation*} 
Hence
\begin{equation*} \sup_{b \in \mathcal{U}^{y}}|\mu(\varphi(x,b)) - \Av(\overline{d})(\varphi(x,b))| = \sup_{q \in S_{y}(M)} |F_{\mu,M}^{\varphi}(q) - F_{\Av(\overline{d}),M}^{\varphi}(q)| < 2\epsilon. 
\end{equation*} 
which completes the proof. 
\end{proof}

\bibliographystyle{amsplain}

\end{document}